\documentclass[11pt]{article}

\newcommand{\ac}[1]{{\color{red}{\bf [~Austin:\ } \color{red}{\em #1}\color{red}{\bf~]}}}

\usepackage{hyperref}
\usepackage[normalem]{ulem}
\usepackage{amsthm}
\usepackage{amsmath}
\usepackage{amssymb}
\usepackage{bm}
\usepackage{mathrsfs}
\usepackage{amsfonts}
\usepackage{mathtools}
\usepackage{graphicx}
\usepackage{tikz}
\usetikzlibrary{arrows}
\usepackage{enumitem}
\usepackage{comment}
\usepackage{caption}
\usepackage{titlesec}
\usepackage{upgreek}
\usepackage{fullpage}
\usepackage{appendix}
\usepackage{algorithm}
\usepackage{algpseudocode}
\usepackage{listings}
\usepackage{xcolor}

\newlist{arrowlist}{itemize}{1}
\setlist[arrowlist]{label=$\to$}

\newcommand{\snsk}[2]{\genfrac{\{}{\}}{0pt}{}{#1}{#2}}

\titleformat{\section}[hang]{\normalfont\bfseries}{\thesection.}{.5em}{}[]
\titlespacing{\section}{\parindent}{3ex plus .1ex minus .2ex}{1em}

\titleformat{\subsection}[runin]{\normalfont\itshape}{\thesubsection.}{.5em}{}[.]
\titlespacing{\subsection}{\parindent}{3ex plus .1ex minus .2ex}{1em}

\newtheorem{theorem}{Theorem}
\newtheorem{conjecture}[theorem]{Conjecture}
\newtheorem{lemma}[theorem]{Lemma}
\newtheorem{proposition}[theorem]{Proposition}

\theoremstyle{definition}
 
\theoremstyle{remark}

\theoremstyle{example}

\numberwithin{theorem}{section}

\providecommand{\R}{}

\providecommand{\N}{}

\providecommand{\P}{}
\renewcommand{\R}{\mathbb{R}}

\renewcommand{\N}{{\mathbb N}}

\renewcommand{\P}{\mathbb{P}}

\newcommand{\unit}{1\!\!1}
\newcommand{\wt}[1]{\widetilde{#1}}
\newcommand{\vep}{\varepsilon}

\newcommand{\E}{\textsf{\upshape E}}
\newcommand{\prob}{\textsf{\upshape Pr}}

\newcommand{\norm}[1]{\left\lVert#1\right\rVert}

\title{Playing Sudoku on random 3-regular graphs}

\author{
Jack Dippel \\
Toronto Metropolitan University \\
\texttt{jack.dippel@torontomu.ca}
\and 
Austin Eide \\
Toronto Metropolitan University\\
\texttt{austin.eide@torontomu.ca}
\and 
Pawe\l{} Pra\l{}at \\
Toronto Metropolitan University\\
\texttt{pralat@torontomu.ca}
\and 
Daniel Willhalm \\
Toronto Metropolitan University\\
\texttt{dwillhalm@torontomu.ca}
} 

\begin{document}

\maketitle

\begin{abstract}
    The \textit{Sudoku} number $s(G)$ of graph $G$ with chromatic number $\chi(G)$ is the smallest partial $\chi(G)$-colouring of $G$ that determines a unique $\chi(G)$-colouring of the entire graph. We show that the Sudoku number of the random $3$-regular graph $\mathcal{G}_{n,3}$ satisfies $s(\mathcal{G}_{n,3}) \leq (1+o(1))\frac{n}{3}$ asymptotically almost surely. We prove this by analyzing an algorithm which $3$-colours $\mathcal{G}_{n,3}$ in a way that produces many \textit{locally forced} vertices, i.e., vertices which see two distinct colours among their neighbours. The intricacies of the algorithm present some challenges for the analysis, and to overcome these we use a non-standard application of Wormald's \textit{differential equations method} that incorporates tools from finite Markov chains.
\end{abstract}

\section{Introduction}

\subsection{Background}

The \emph{Sudoku number} $s(G)$ of a graph $G = (V,E)$ with chromatic number $\chi(G)$ is the smallest possible size of a subset $S \subseteq V$ with the following property: there is some partial $\chi(G)$-colouring of the induced subgraph $G[S]$ that has a unique extension to a $\chi(G)$-colouring of all of $G$. We call any set with this property a \emph{Sudoku set} for $G$. While our terminology follows that of the recent work \cite{lau2023sudoku}, Sudoku sets have previously appeared in the literature as \textit{defining sets}~(\cite{mahmoodian1998defining}--\cite{mahmoodian2005defining}) and \textit{determining sets}~(\cite{cambie2024bounds,cooper2014critical}). (Generally, \textit{critical set} has been used to refer to a minimal Sudoku set.) The parameter $s(G)$ was also called the \textit{forcing chromatic number} in~\cite{harary2007computational}. Determining $s(G)$ is a generalization of the analogous problem for Latin squares---namely, identifying a smallest set of entries that uniquely determines a full Latin square---which dates to~\cite{nelder1977critical}.

The study of Sudoku sets is motivated by problems in combinatorics, optimization, and information theory, where identifying minimal structures that uniquely determine global properties is of significant interest. For instance, in~\cite{hatami2016teaching} Hatami and Qian relate Sudoku sets in Latin Squares to the concepts of \textit{teaching dimension} and \textit{VC dimension}, which have important applications in machine learning. More popularly, in~\cite{mcguire2014there} it was established (with computer assistance) that ``there is no $16$-clue Sudoku," confirming a conjecture that $17$ clues are needed to determine a Sudoku puzzle with a unique solution. (The name \textit{Sudoku number} is, of course, inspired by this problem.)

Most previous work on the Sudoku number has focused primarily on proving bounds on $s(G)$ for specific families of graphs. We review some results which are relevant for our purposes, beginning with some lower bounds. In~\cite{mahmoodian1997defining}, it is shown that for any graph $G$, 
\begin{equation}\label{eq:sudoku_deterministic_lower_bound1}
    s(G) \ge |V(G)| - \frac {|E(G)|}{\chi(G)-1}.
\end{equation}
\cite{mahmoodian1999defining} studies the Sudoku number for $d$-regular graphs with chromatic number $d$, in which case
\begin{equation}\label{eq:sudoku_determininstic_lower_bound2}
    s(G) \ge \bigg\lceil \frac{d-2}{2(d-1)} |V(G)| + \frac{2 + (d-2)(d-3)}{2(d-1)}\bigg\rceil.
\end{equation}
(We remark that~\eqref{eq:sudoku_determininstic_lower_bound2} gives a lower bound of $\left\lceil \frac{|V(G)|}{4} + \frac{1}{2} \right\rceil$ when $d = 3.$)

For $d$-regular graphs $G$ with $d\ge 4$ and $\chi(G) = 3$, it is also shown in \cite{mahmoodian1999defining} that $s(G) \ge 2$. Further, \cite{mahmoodian2006defining} proves that for any $d\ge 3$ and any at least $2(d-1)$-regular graph $G$ with $\chi(G) = d$ and $|V(G)| \ge 3d$ it holds that $s(G) \ge d-1$. All of these lower bounds are sharp with the exception of \eqref{eq:sudoku_determininstic_lower_bound2}, where sharpness requires $d\in\{3,4,5\}$. Upper bounds on $s(G)$ are harder to come by. It is known (see for instance \cite[Theorem 2.1]{cambie2024bounds}) that $s(G) \leq n-1$ for any connected graph $G$ on $n$ vertices, with equality if and only if $G = K_{n}$. However, beyond this, the literature contains relatively little on general upper bounds for $s(G)$, notwithstanding some results on specific graph families in, e.g., \cite{cambie2024bounds}, \cite{cooper2014critical}, and \cite{lau2023sudoku}.

Notably, the Sudoku number of random graphs has not been studied yet. The random $d$-regular graph $\mathcal{G}_{n,d}$ provides a natural model for studying this problem, as the chromatic number $\chi(\mathcal{G}_{n,d})$ is known to concentrate on a single (explicit) value \textit{asymptotically almost surely}, or a.a.s., for many fixed values of $d$ as $n \to \infty$ \cite{kemkes2010chromatic}. (See Subsection \ref{sec:models} for a formal definition of the probability space and more on asymptotic notation.) In this paper, we focus our attention on bounding the Sudoku number of the random cubic graph $\mathcal{G}_{n,3}$, which is known to have chromatic number $3$ a.a.s. Our main result is the following.

\begin{theorem}\label{thm:main_theorem}
    A.a.s., $s(\mathcal{G}_{n,3}) \leq (1+o(1))\frac{n}{3}$.
\end{theorem}

To prove Theorem~\ref{thm:main_theorem} we design an algorithm \textbf{Sudoku} which, a.a.s., properly $3$-colours $\mathcal{G}_{n,3}$ and simultaneously constructs a corresponding Sudoku set of size at most $(1+o(1))\frac{n}{3}.$ (In fact this is not quite true---our algorithm colours a random graph that is \textit{contiguous} with $\mathcal{G}_{n,3}$ which allows us to exploit more structure; see Subsection~\ref{sec:models}.) Most of the paper is devoted to analyzing this algorithm, a task to which we apply the \textit{differential equations method} of Wormald~\cite{wormald1999differential}. 

Despite some effort, we were unable to improve the deterministic lower bound $s(\mathcal{G}_{n,3}) \geq \lceil n/4 + 1/2 \rceil$ implied by~\eqref{eq:sudoku_determininstic_lower_bound2}. Perhaps surprisingly, some existing results, supported by simulations, seem to suggest that this lower bound may be correct, and we conjecture as much below.

\begin{conjecture}\label{conj:n/4}
    A.a.s., $s(\mathcal{G}_{n,3}) = (1+o(1))\frac{n}{4}.$
\end{conjecture}

We cite two previous results which lend support to Conjecture~\ref{conj:n/4}. The first, from \cite{mahmoodian1999defining}, requires an additional definition: a Sudoku set $S$ is \textit{strong} if there exists an ordering $v_{1}, v_{2}, \dots, v_{|V(G)| - |S|}$ of $|V \setminus S|$ such that for each $i \in \{1,\dots, |V(G)| - |S|\}$, vertex $v_{i}$ has neighbours of all but one colour  in $S \cup \{v_{1}, \dots, v_{i-1}\}.$ Strong Sudoku sets thus force colours as locally as possible: if we colour vertices in $V(G) \setminus S$ sequentially in the order $v_{1}, v_{2}, \dots, v_{|V(G)| - |S|}$, then at each step $i$ the colour on vertex $v_{i}$ is forced. The following theorem explains the relevance of strong Sudoku sets to our problem.

\begin{theorem}[\protect{\cite[Lemma 1]{mahmoodian1999defining}}]\label{thm:strong_sudoku}
    If $G$ is $d$-regular and $\chi(G) = d$, then every Sudoku set for $G$ is strong.
\end{theorem}

Note that if $S$ is a strong Sudoku set for a $d$-regular graph with $\chi(G) = d$, then $G \setminus S$ induces a forest. Indeed, if there were a cycle in $G \setminus S$ containing some vertex $v_{i}$, then $v_{i}$ has at most $d-2$ neighbours in $S \cup \{v_{1}, \dots, v_{i-1}\}$, meaning $S$ is not strong. In particular, any Sudoku set in this case is a \textit{decycling set}, i.e., a set of vertices whose removal from $G$ destroys all cycles. Theorem~\ref{thm:strong_sudoku} suggests decycling sets make good candidates for Sudoku sets.

Decycling sets for $\mathcal{G}_{n,d}$ have been previously studied by Bau, Wormald, and Zhao, who show in particular the following for $\mathcal{G}_{n,3}$:

\begin{theorem}[\protect{\cite[Theorem 1.1]{wormald2001decycling}}]\label{thm:decycling}
    A.a.s., $\mathcal{G}_{n,3}$ has a decycling set of size $\left\lceil \frac{n}{4} + \frac{1}{2} \right\rceil$, which is smallest possible for a cubic graph on $n$ vertices.
\end{theorem}

The proof of Theorem~\ref{thm:decycling} provides a polynomial-time algorithm which, a.a.s., finds a smallest decycling set in a random graph obtained by taking the union of a Hamilton cycle and a random perfect matching on $n$ vertices. (This random graph model is contiguous with the random $3$-regular graph---see Subsection~\ref{sec:models} for formal definitions.) This allowed us to run simulations. We conducted an experiment in which we independently generated many random $3$-regular graphs on $60$ vertices, built from a Hamilton cycle and random matching as above. For roughly half of them, we were able to find a colouring of the decycling set that implied a unique extension to the whole graph. If none was found, it was almost always possible to extend the closest candidate to a Sudoku set by adding a single vertex to it. Figure~\ref{fig:enter-label} illustrates this with an example of a $3$-regular graph on $70$ vertices, where the Sudoku set has size $18$, the smallest possible value. These observations inspired us to make Conjecture~\ref{conj:n/4}. 

\begin{figure}[h]
    \centering
    \includegraphics[width=0.49\linewidth]{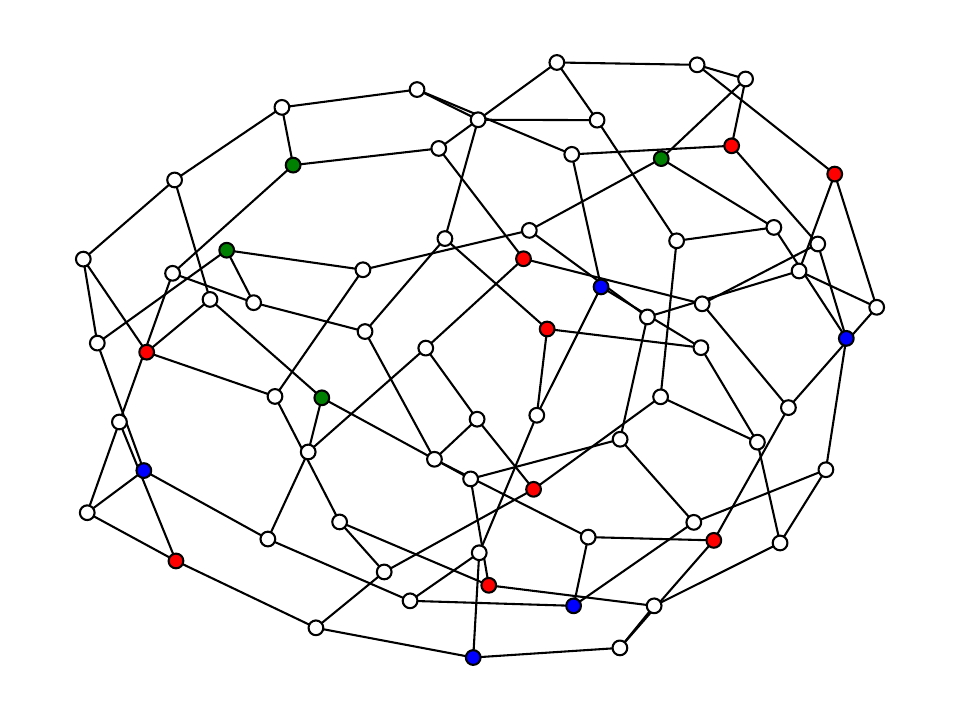}
    \includegraphics[width=0.49\linewidth]{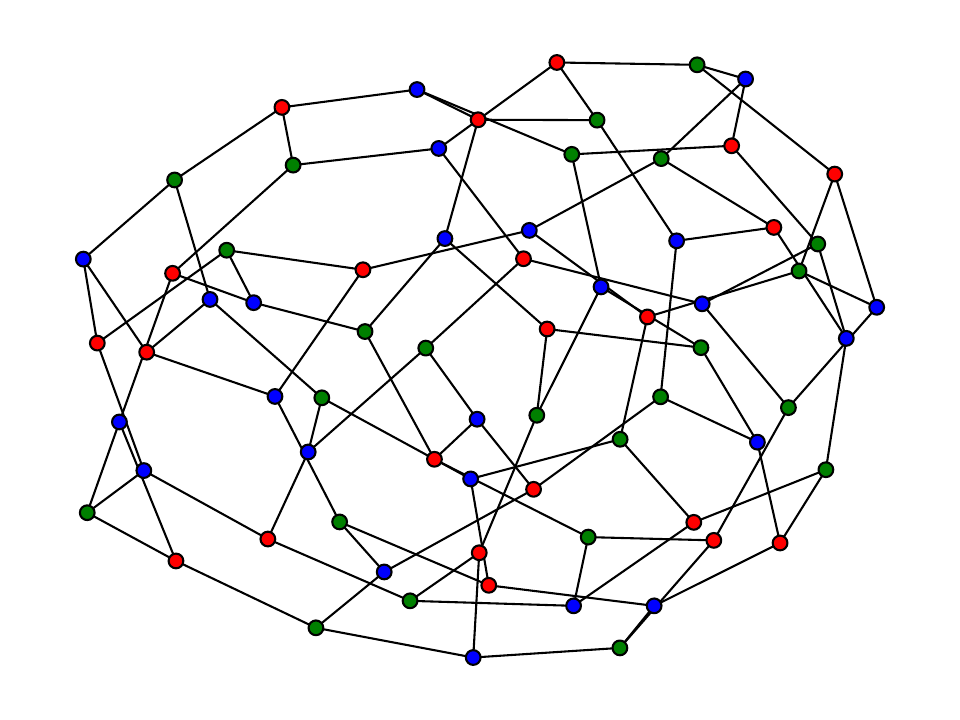}
    \caption{The left graph illustrates a Sudoku set of a random $3$-regular graph with $70$ vertices, generated using a Hamilton cycle and a random perfect matching. The Sudoku set shown is of minimum size for this graph. The right graph shows the proper colouring of the whole graph determined by the Sudoku set.}
    \label{fig:enter-label}
\end{figure}

\subsection{On models and contiguity}\label{sec:models}

Our results are asymptotic by nature, that is, we will assume that $n \to \infty$. We say that an event holds \emph{asymptotically almost surely} (\emph{a.a.s.}) if it holds with probability tending to one as $n \to \infty$.

For $d \geq 1$, we let $\mathcal{G}_{n,d}$ be the uniform probability space on all $d$-regular graphs with vertex set $[n]$. (Often, we will abuse notation and speak of $\mathcal{G}_{n,d}$ itself as a random $d$-regular graph.) The model $\mathcal{G}_{n,d}$ is often studied via a more general \textit{configuration model} $\mathcal{G}_{n,d}^{*}$, introduced by Bollob{\'a}s in~\cite{bollobas1980probabilistic}, which is a probability space over $d$-regular multigraphs on vertex set $[n]$. $\mathcal{G}_{n,d}^{*}$ enjoys the following properties with respect to $\mathcal{G}_{n,d}$:
    \begin{enumerate}[label = (\roman*)]
        \item for fixed $d$, the probability that $\mathcal{G}_{n,d}^{*}$ is simple approaches a positive constant, dependent on $d$ but not on $n$, as $n \to \infty$, and;
        \item conditioned on being simple, $\mathcal{G}_{n,d}^{*}$ is distributed uniformly over all simple $d$-regular graphs on vertex set $[n]$, i.e., conditioned on being simple $\mathcal{G}_{n,d}^{*}$ has the same distribution as $\mathcal{G}_{n,d}$.
    \end{enumerate}
Thus, to show that some graph property $\mathcal{A}_{n}$ holds a.a.s.\ in $\mathcal{G}_{n,d}$, it suffices to show that a corresponding multigraph property $\mathcal{A}_{n}^{*}$ holds a.a.s.\ for $\mathcal{G}_{n,d}^{*}$. This is essentially the approach we take to prove Theorem~\ref{thm:main_theorem}, though we make a small modification to incorporate a more convenient model. Let $\mathcal{G}_{n,d}'$ be the random multigraph $\mathcal{G}_{n,d}^{*}$ conditioned to have no self-loops (though with multiedges still allowed). Then by \cite[(3.3)]{janson1995random}, $\mathcal{G}_{n,3}'$ is \textit{contiguous} with the random multigraph $G_{n}$ defined as the union of a random Hamilton cycle $H_{n}$ on vertex set $[n]$ and a random perfect matching $M_{n}$ on $[n]$. Contiguity here means that an event holds a.a.s.\ for $\mathcal{G}_{n,3}'$ if and only if it holds a.a.s.\ for $G_{n}.$ 

Now, properties (i) and (ii) above are equally valid when $\mathcal{G}_{n,d}^{*}$ is replaced with $\mathcal{G}_{n,d}'$. Therefore, using contiguity of $\mathcal{G}_{n,3}'$ and $G_{n}$, we have the following: if $\mathcal{A}_{n}^{*}$ is some multigraph property which restricts to a graph property $\mathcal{A}_{n}$ on simple graphs, then if $\mathcal{A}_{n}^{*}$ holds a.a.s.\ in $G_{n}$, $\mathcal{A}_{n}$ holds a.a.s.\ in $\mathcal{G}_{n,3}$. We will show that a.a.s., $G_{n}$ has a proper $3$-colouring $c$ which is determined by its values on a set $S \subseteq [n]$ of size $|S| \leq (1+o(1))n/3$. By our reasoning above, the same holds a.a.s.\ for $\mathcal{G}_{n,3}$, and this, along with the well-known fact that $\chi(\mathcal{G}_{n,3}) = 3$ a.a.s., implies Theorem \ref{thm:main_theorem}.

\subsection{Concentration Inequalities}

We will make use of two-well known concentration inequalities. The first is a generalization of the Chernoff bound to sums of negatively correlated $\{0,1\}$-random variables. A collection $Y_{1}, Y_{2},\dots,Y_{n}$ of $\{0,1\}$-random variables is \textit{negatively correlated} if for all $J \subseteq [n]$, 
$$
    \prob(Y_{j} = 1\, \forall \, j \in J) \leq \prod_{j \in J}\prob(Y_{j} = 1).
$$
Note that collections of independent $\{0,1\}$-random variables are negatively correlated under this definition.

\begin{proposition}[Chernoff bounds with negative correlation, \protect{\cite[Theorems 1.10.23 and 1.10.24]{Doerr2020}}]\label{prop:chernoff}
    Let $Y_{1},Y_{2},\dots,Y_{n}$ be $\{0,1\}$-random variables such that both $\{Y_{j}\}_{j=1}^{n}$ and $\{1-Y_{j}\}_{j=1}^{n}$ are negatively correlated. Let $Y = \sum_{j=1}^{n}Y_{j}.$ Then for any $\lambda > 0$,
    $$
        \prob(Y \leq \E[Y] - \lambda) \leq e^{-\frac{\lambda^{2}}{2\E[Y]}}
    $$
    and 
    $$
        \prob(Y \geq E[Y] + \lambda) \leq e^{-\frac{\lambda^{2}}{2(\E[Y] + \lambda/3)}}.
    $$
    Consequently, if $0 \leq \lambda \leq \frac{3\E[Y]}{2}$, then 
    $$
        \prob(|Y - \E[Y]| \geq \lambda) \leq 2e^{-\frac{\lambda^{2}}{3\E[Y]}}.
    $$
\end{proposition}


We also use Azuma's inequality to control the large deviations of sub-and super-martingales.

\begin{proposition}[Azuma's inequality, \protect{\cite[Lemma 4.1]{wormald1999differential}}]\label{prop:azuma}
    Suppose $M_{0}, M_{1}, M_{2},\dots$ is a submartingale such that $|M_{j+1} - M_{j}| \leq C$ for all $j \geq 0$. Then for any $\lambda \geq 0$
    $$
        \prob(M_{n} \leq M_{0} - \lambda) \leq e^{-\frac{\lambda^{2}}{2C^{2}n}}.
    $$
    If $M_{0}, M_{1}, M_{2},\dots$ is a super-martingale such that $|M_{j+1}-M_{j}| \leq C$ for all $j \geq 0$, then for any $\lambda \geq 0$
    $$
        \prob(M_{n} \geq M_{0} + \lambda) \leq e^{-\frac{\lambda^{2}}{2C^{2}n}}.
    $$
\end{proposition}

\subsection{A word on floors and ceilings}

As is typical in the field of random graphs, for expressions which clearly must be integers (e.g., indices in the set $[n]$) we round up or down to the nearest integer but do not specify which. In all cases that we use this simplification, it will be clear that the choice of which way to round has no significant effect on the final computation.  

\section{The algorithm \textbf{Sudoku}}\label{sec:the_algorithm}

Here we introduce our primary algorithm, \textbf{Sudoku}, which, given a multigraph $G$ that is the union of a Hamilton cycle $H$ and perfect matching $M$ on vertex set $[n]$, produces a proper $3$-colouring $c$ of $G$ along with a Sudoku set $S$ for $c$. \textbf{Sudoku} is a variant of a simple greedy colouring algorithm which we describe below. For now, we state everything in terms of deterministic graphs; later, when analyzing \textbf{Sudoku} on the random multigraph $G_{n} = H_{n} \cup M_{n}$, we will couple the execution of the algorithm with a random graph process which reveals the edges of $M_{n}$ one-at-a-time.  

\subsection{Greedily colouring Hamiltonian cubic graphs}\label{sec:greedy_alg}

For simplicity, we assume $H$ has cyclic ordering $(12\cdots n)$. For any $i \in [n]$, we let $p(i)$ be the unique $j \in [n]$ such that $\{i,j\} \in M$---i.e., $j$ is the partner of $i$ in $M$. The generic greedy colouring algorithm  sequentially assigns colours in $\{1,2,3\}$ to each vertex $i \in [n]$, starting from $1$ and proceeding in the order of $H$.

Begin by selecting $c(1)$ from $\{1,2,3\}$. For $i \in \{2,3,\dots,n-1\}$, we then select $c(i)$ from $\{1,2,3\} \setminus \{c(i-1), c(p(i))\}$ if $p(i) < i$, or from $\{1,2,3\} \setminus \{c(i-1)\}$ if $p(i) > i$. The selection at each step can be made randomly, deterministically, with a mixture of the two, etc. However the selections are made, it is clear that this process constructs a partial proper $3$-colouring $c$ of the vertices in $[n-1]$. When the process reaches vertex $n$, it may be the case that $n$ has neighbours in three distinct colour classes, which prevents us from completing the colouring. 

There are a variety of ways one could attempt to circumvent this issue. We propose one which is well-suited to our needs in the random setting. Suppose that for some interval of vertices $I = \{i, i+1, \dots, i+2j-1\}$ with $1 \le i < i+2j-1 \le n$, we have $p(i) = i+ 2j - 1$ and $p(i') \not\in I$ for all $i' \in I \setminus \{i, i+2j-1\}$. Then, the subgraph of $G$ induced by $I$ is an even cycle. We may greedily colour the vertices in $[n] \setminus I$ as follows. We first greedily colour the interval $\{1,2,\dots, i-1\}$. (Note that if $i = 1$, this interval is empty and thus there are no vertices to colour.) We then greedily colour the interval $\{i+2j, i+ 2j+1, \dots n\}$ in reverse order (this interval is empty if $i+2j-1 = n$). Since $I$ induces a cycle in $G$, after colouring $[n] \setminus I$ each vertex in $I$ has exactly one neighbour that is coloured, leaving $2$ colours available. It is well-known (see, e.g., \cite{erdos1979choosability}) that even cycles are $2$-list colourable, and hence the colouring can be extended to $I$.

When colouring the random graph $G_{n}$, we will show that a.a.s.\ there is some interval of consecutive vertices in $\{n - n/\log\log n,\dots, n\}$ which induces an even cycle. Then, we will carefully colour vertices $1$ through $n - n/\log\log n$, identifying vertices that need to be put into the Sudoku set. Once we are done, the above property ensures that a.a.s.\ the coloring can be completed. During this \textit{completion phase}, we may simply put all of the last $n/\log\log n$ vertices into the Sudoku set, since they will not affect its asymptotic size.

\subsection{\textbf{Sudoku}}\label{sec:sudoku_alg}

Here we outline \textbf{Sudoku}. While the algorithm is somewhat complicated, it is nonetheless a direct instance of the greedy algorithm defined in the previous section. \textbf{Sudoku} sequentially builds a Sudoku set $S \subseteq [n]$ in addition to the colouring $c$. We will define \textbf{Sudoku} to start at a general vertex $i_{0} \in [n]$ and end at a general vertex $i_{1}$ satisfying $i_{0} \leq i_{1} \leq n-1$. We assume that vertices in $[i_{0}]$ are properly $3$-coloured by some partial colouring $c_{0}$, and that there is a set $S_{0} \subseteq [i_{0}]$ that is a Sudoku set for $c_{0}$ on the subgraph induced by vertices in $[i_{0}].$ 

\textbf{Sudoku} constructs a colouring $c$ that extends $c_{0}$ and a set $S$ that extends $S_{0}$. We let $S(i_{0}) = S_{0}$, and for $i \in \{i_{0}+1, \dots, i_1\}$ let $S(i)$ be the set constructed by the algorithm at time $i$, i.e., just after vertex $i$ is processed. We also maintain a pointer which at time $i$ points to some vertex $j \in [i]$---we denote this by $\textbf{ptr}(i) = j$. The pointer will maintain the property that if $\textbf{ptr}(i) = j$, then the set $S(i) \cap [j]$ is a Sudoku set for the colouring $c$ restricted to the subgraph induced by vertices in $[j]$. Thus, at any time $i$ during the algorithm, any vertex in $[\textbf{ptr}(i)]$ is either in $S(i) \cap [\textbf{ptr}(i)]$, or has its colour forced by other vertices in $S(i) \cap [\textbf{ptr}(i)]$. We necessarily have $\textbf{ptr}(i_{0}) = i_{0}$, as we assume that $S(i_{0}) = S_{0}$ is a Sudoku set for the colouring on $[i_{0}]$.

The algorithm proceeds in an alternating sequence of two types of \textit{runs}, where the type of the run at time $j$ is determined by $\textbf{ptr}(i)$, i.e., the location of the pointer with respect to vertex $i$. Each type of run has its own colouring rules, and a run of one type ends (and a run of the other type begins) when we encounter a \textit{bad} vertex, where \textit{bad} is also defined differently for each type of run. We give full details below.

\vspace{0.5cm}

\textbf{Runs of type $A$}

A run of type $A$ is defined as a sequence of consecutive vertices $i, i+1, \dots, i+r$ satisfying $\textbf{ptr}(j) = j$ for all $j \in \{i, i+1, \dots, i+r\}$. Suppose that we have processed all vertices in $[i]$ and that we are in a run of type $A$, i.e., $\textbf{ptr}(i) = i$. There are two possibilities for $i+1$, each with corresponding update rules:

\begin{enumerate}[label = ($A$\arabic*)]
    \item $i+1$ is \textit{good} if $p(i+1) < i+1$ and $c(p(i+1)) \neq c(i)$.

        \begin{arrowlist}
            \item Assign $c(i+1) = k$ for $k$ the unique colour not in $\{c(p(i+1)), c(i)\}$;
            \item set $S(i+1) = S(i)$;
            \item set $\textbf{ptr}(i+1) = i+1.$
        \end{arrowlist}

    \item $i+1$ is \textit{bad} if:
        \begin{enumerate}[label = (\alph*)]
            \item $p(i+1) > i+1$; or
            \item $p(i+1) < i+1$ and $c(p(i+1)) = c(i)$.
        \end{enumerate}
        \begin{arrowlist}
            \item Choose $c(i+1)$ randomly from $\{1,2,3\} \setminus \{c(i)\}$;
            \item set $S(i+1) = S(i)$;
            \item set $\textbf{ptr}(i+1) = \textbf{ptr}(i) = i$.
        \end{arrowlist}
\end{enumerate}

Thus a good vertex in a run of type $A$ extends the run, while a bad vertex ends it and becomes the first vertex of a run of type $B$. Note that a good vertex has its colour forced by vertices in $[i]$, as the condition $\textbf{ptr}(i) = i$ ensures that each of $i$ and $p(i+1) \in [i]$ is either already in $S(i)$, or has its colour forced by other vertices in $[i]$. (This is proved formally in Subsection~\ref{sec:correctness}.)

\vspace{0.5cm}

\textbf{Runs of type $B$}

A run of type $B$ is a sequence $i, i+1, \dots, i+r$ with $\textbf{ptr}(j) < j$ for all $j \in \{i, i+1, \dots, i+r\}$. Assume we have processed through $i$ and $\textbf{ptr}(i) < i$, so we are in a run of type $B$. Define the set $C_{i+1}$ by 
$$
    C_{i+1} = \begin{cases} \{c(i-1), c(i)\} \quad& p(i) > i\\
    \{c(i), c(p(i))\} \quad& p(i) < i.
    \end{cases}
$$
(Provided that $c$ is a proper colouring on $[i]$, we have $|C_{i+1} = 2|$.) The possibilities for $i+1$ are:

\begin{enumerate}[label = ($B$\arabic*)]
    \item $i+1$ is \textit{good} if $p(i+1) \leq \textbf{ptr}(i)$ and $c(p(i+1)) \in C_{i+1}.$
        \begin{arrowlist}
            \item Assign $c(i+1) = k$ for $k$ the unique colour not in $C_{i+1}$;
            \item set $S(i+1) = S(i)$;
            \item set $\textbf{ptr}(i+1) = \textbf{ptr}(i)$.
        \end{arrowlist}
    \item $i+1$ is \textit{bad} if 
        \begin{enumerate}[label = (\alph*)]
            \item $p(i+1) > i+1$; or
                \begin{arrowlist}
                    \item Assign  $c(i+1) = k$ for $k$ the unique colour not in $C_{i+1}$;
                    \item set $S(i+1) = S(i) \cup \{i+1\}$;
                    \item set $\textbf{ptr}(i+1) = i+1$.
                \end{arrowlist}
            \item $p(i+1) \leq \textbf{ptr}(i)$ and $c(p(i+1)) \not\in C_{i+1}$; or
                \begin{arrowlist}
                    \item Assign $c(i+1) = k$ for $k$ the unique colour not in $\{c(i), c(p(i+1))\}$;
                    \item set $S(i+1) = S(i) \cup \{i\}$;
                    \item set $\textbf{ptr}(i+1) = i+1$.
                \end{arrowlist}
            \item $\textbf{ptr}(i) + 1 \leq p(i+1) \leq i$ and $c(p(i+1)) \in C_{i+1}$; or
                \begin{arrowlist}
                    \item Assign $c(i+1) = k$ for $k$ the unique colour not in $C_{i+1}$;
                    \item set $S(i+1) = S(i) \cup \{i+1\}$;
                    \item set $\textbf{ptr}(i+1) = i+1.$
                \end{arrowlist}
            \item $\textbf{ptr}(i) +1 \leq p(i+1) \leq i$ and $c(p(i+1)) \not\in C_{i+1}$
                \begin{arrowlist}
                    \item Assign $c(i+1) =  k$ for $k$ the unique colour not in $\{c(i), c(p(i+1))\}$;
                    \item set $S(i+1) = S(i) \cup \{i, i+1\}$;
                    \item set $\textbf{ptr}(i+1) = i+1.$
                \end{arrowlist}
        \end{enumerate}
\end{enumerate}

As with runs of type $A$, a good vertex in a run of type $B$ extends the run, while a bad vertex ends it (and begins a run of type $A$). A run of type $B$ can be interpreted as a second chance: after a run of type $A$ ends with a bad vertex starting a run of type $B$, we get another attempt to force the colour of this vertex in the run of type $B$. We remark that any run of either type can have at most one vertex with a forward edge, and that vertex can only be the first vertex in the run. Indeed, vertices with forward edges are always bad, and thus they necessarily end the current run and become the first vertex of the next one. In particular, this means that in cases ($B2$c) and ($B2$d) above, we necessarily have $p(i+1) = \textbf{ptr}(i+1).$

Colours on vertices in a run of type $B$ are (eventually) forced by vertices with larger indices. Suppose that $i$ is in a run of type $B$, and that $i+1$ is good. Then we colour $i+1$ with the unique colour $k$ not in $C_{i+1}$, at which point vertex $i$ sees two different colours among its neighbours, one of which must be in $[\textbf{ptr}(i)]$; thus if $i+1$ is somehow forced to have colour $k$, $i$ will be forced to have colour $c(i)$. The colouring rules in case ($B2$) ensure that the chain of dependencies eventually stops when we reach a bad vertex, and that we may colour appropriately so that all vertices in the run have their colours forced. 

\subsection{Correctness of the algorithm and bound on the size of $S$}\label{sec:correctness}

It is clear that \textbf{Sudoku} produces a partial proper $3$-colouring $c$ which extends $c_{0}$ to vertices in $\{i_{0}, i_{0}+1, \dots,i_{1}\}$. Here we show the following.

\begin{lemma}\label{lem:correctness}
    Let $c_{0}$ be a proper partial $3$-colouring of the subgraph of $G$ induced by $[i_{0}]$, let $S_{0}$ be a Sudoku set for $c_{0}$ on $[i_{0}]$, and let \textup{\textbf{Sudoku}} run from $i_{0}$ (with $S(i_{0}) = S_{0}$) until vertex $i_{1} \leq n-1$. Let $c$ be the extension of $c_{0}$ produced by \textup{\textbf{Sudoku}}. Then $S(i_{1})$ is a Sudoku set for $c$ restricted to the subgraph induced by $[\textup{\textbf{ptr}}(i_{1})]$.
\end{lemma}

In particular, if we set $S = S(i_{1}) \cup \{i_{1}\}$, then $S$ is a Sudoku set for the colouring $c$ restricted to the subgraph induced by $[i_{1}]$.

\begin{proof}
    We let $\mathcal{P}_{i}$ be the property that $S(i)$ is a Sudoku set for $c$ on the subgraph induced by $[\textbf{ptr}(i)]$. We show inductively that $\mathcal{P}_{i_{1}}$ holds. Note that $\mathcal{P}_{i_{0}}$ holds by assumption. Now, let $i \geq i_{0}$ and suppose that $\mathcal{P}_{i}$ holds. It suffices to assume that for vertex $i+1$ we are in case ($A1$) or case ($B2$), as these are the only cases in which the pointer moves. Verifying that $\mathcal{P}_{i+1}$ holds for case ($A1$) is straightforward. We thus assume case ($B2)$, i.e., that $\textbf{ptr}(i) < i$ and that $i+1$ is bad. In this case, the pointer moves to $i+1$, so we must show that $S(i+1)$ is a Sudoku set for $c$ on the subgraph induced by $[i+1].$

    Note that every $j \in \{\textbf{ptr}(i) + 1, \dots, i\}$ has a neighbour in $[\textbf{ptr}(i)]$---we have $p(j) \in [\textbf{ptr}(i)]$ for $j > \textbf{ptr}(i) + 1$, and $\textbf{ptr}(i)$ itself is adjacent to $j$ for $j = \textbf{ptr}(i) + 1$. Further, for each $j \in \{\textbf{ptr}(i) + 1,\dots,i-1\}$, $c(j+1)$ is distinct from the neighbour of $j$ that is in $[\textbf{ptr}(i)]$. This holds because for each such $j$, we colour $j+1$ as in case ($B1$). Thus it follows that if the colour on vertex $i$ is forced, the colours on all $j \in \{\textbf{ptr}(i)+1, \dots, i\}$ are forced as well.

    In case ($B2$a), $c(i)$ is forced because we colour $i+1$ with a colour which is not yet in the neighbourhood of $i$ and include $i+1$ in $S(i+1)$. (Vertex $i+1$ is thus trivially forced in this case.) In ($B2$a), we include $i$ itself in $S$; in this case vertex $i+1$ is forced because $c(p(i+1))$ is in $[\textbf{ptr}(i)]$ and is distinct from $c(i).$ Case ($B2$c) is effectively equivalent to ($B2$a). Clearly ($B2$d) holds, since both $i$ and $i+1$ are included in $S$ in this case.

    In all cases, we have that $\mathcal{P}_{i+1}$ holds, and thus we conclude by induction that we have $\mathcal{P}_{i_{1}}$. \end{proof}

We also obtain an upper bound on the size of $S(i_{1})$. For a run of \textbf{Sudoku} from $i_{0}$ to $i_{1}$, we define $\mathcal{B}_{C} \subseteq \{i_{0}, \dots, i_{1}\}$ to be the set of \textit{conventionally bad} vertices: these are vertices $j$ which, when processed by \textbf{Sudoku}, fall into cases ($A2$a), ($A2$b), ($B1$a), or ($B1$b). The set $\mathcal{B}_{C}$ includes all $j$ with forward edges and all $j$ with backward edges which fail to extend a run because $p(j)$ has a ``forbidden" colour, i.e., $p(j)$ has the unique colour we would like to use on $j$.

Let $\mathcal{B}_{U}^{(c)}$ and $\mathcal{B}_{U}^{(d)}$ consist of the the vertices $j$ covered by ($B2$c) and ($B2$d), respectively. These vertices are \textit{unconventionally bad}, and correspond to the situation in which we are in a run of type $B$ and $\textbf{ptr}(j-1) + 1 \leq p(j) \leq j-1.$

\begin{lemma}\label{lem:size_bound}
    We have 
    $$
        |S(i_{1})| \leq \frac{1}{2}|\mathcal{B}_{C}| + |\mathcal{B}_{U}^{(c)}| + 2|\mathcal{B}_{U}^{(d)}| + |S_{0}|.
    $$
\end{lemma}

\begin{proof}
    Note that $|S(i)|$ only increases when we encounter a bad vertex in the algorithm. We will calculate the contribution from each kind of bad vertex. 
    
    Clearly we add one vertex to $S$ for each vertex in $\mathcal{B}_{U}^{(c)}$ and two vertices to $S$ for each vertex in $\mathcal{B}_{U}^{(d)}.$ We also add one vertex to $S$ for every vertex in $\mathcal{B}_{C}$ which ends a run of type $B$; these vertices make up at most half of $\mathcal{B}_{C}$ because runs alternate between types $A$ and $B$, and we begin the algorithm in a run of type $A$. 
\end{proof}

We could be more careful in the proof of Lemma~\ref{lem:size_bound} and derive an exact expression for $|S(i_{1})|$, but the upper bound we give will suffice for our purposes. 

\subsection{Vertex types}

We can think of the \textbf{Sudoku} as making transitions among a finite set of vertex types, where a type encodes some characteristics of a particular vertex when it is processed by the algorithm, e.g., its colour, the colour of its neighbours, whether it has a forward or backward edge, etc. 

Suppose a vertex $i$ has the following properties: $i$ is assigned colour $k$, $i$ has a backward edge, and $\textbf{ptr}(i) = i$ in the execution of Sudoku. Then we say $i$ is of type $A_{b}^{(k)}$, where $A$ refers to the fact that a vertex with these properties is part of a run of type $A$. Similarly, $i$ is of type $A_{f}^{(k)}$ if $i$ has colour $k$, $i$ has a forward edge, and $\textbf{ptr}(i) = i$. We let 
$$
    \mathcal{A} = \{A_{b}^{(1)}, A_{b}^{(2)}, A_{b}^{(3)}, A_{f}^{(1)}, A_{f}^{(2)}, A_{f}^{(3)} \}
$$
be the formal class consisting of these types.  

Now, suppose that vertex $i$ has colour $k$, $i$ has a backward edge to some vertex of colour $\ell$, and $\textbf{ptr}(i) < i$ in the execution of \textbf{Sudoku}. Then we say $i$ is of type $B_{b}^{(\ell k)}$. If instead $i$ has a forward edge and $i-1$ is of colour $\ell$, $i$ is of type $B_{f}^{(\ell k)}$. We define
$$
    \mathcal{B} = \{B_{b}^{(12)}, B_{b}^{(21)}, B_{b}^{(13)}, B_{b}^{(31)}, B_{b}^{(23)}, B_{b}^{(32)}, B_{f}^{(12)}, B_{f}^{(21)}, B_{f}^{(13)}, B_{f}^{(31)}, B_{f}^{(23)}, B_{f}^{(32)}\}
$$
and let $\mathcal{V} = \mathcal{A} \cup \mathcal{B}$ be the set of all vertex types. See Figure~\ref{fig:types} for an illustration of the types.

\begin{figure}[h]
	\centering
	\includegraphics[scale=0.375]{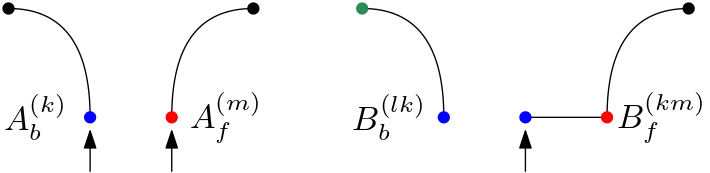}
	\caption{The types. Here we use blue for colour $k$, red for $m$, and green for $\ell$. The pointer is indicated by an upward arrow. Note that if vertex $i$ is of type $B_{f}^{(km)}$, we must have $\textbf{ptr}(i) = i-1$ as indicated in the right-most figure.}
	\label{fig:types}
\end{figure}

Transitions between the types are governed by the algorithm \textbf{Sudoku}. For instance, if $i$ has type $A_{b}^{(1)}$ and vertex $i+1$ has a backward edge to a vertex of colour $2$, then $i+1$ has type $A_{b}^{(3)}$. If instead vertex $i+1$ has a forward edge, then its type will be $B_{f}^{(12)}$ with probability $\frac{1}{2}$ and $B_{f}^{(13)}$ with probability $\frac{1}{2}$, as in this situation we colour $i+1$ randomly with some colour in $\{2,3\}$. In Section~\ref{sec:markov_chain}, we will design a Markov chain on the set of types $\mathcal{V}$ that mimics these transition dynamics. 

\subsection{The random matching process}\label{sec:random_matching_process}

In order to analyze \textbf{Sudoku} on the random graph $G_{n}$, we employ a common technique: we reveal the edges of the graph and ``run the algorithm" simultaneously. (More formally, this is known as the \textit{principle of deferred decision}.) To start, we reveal the random Hamilton cycle $H_{n}$; since we do not care about the vertex labels, we may always relabel $H_{n}$ so that it has the convenient cyclic order $(12\cdots n)$. (We will assume that this is the case for the remainder of the paper.)

We will sample the matching $M_{n}$ via a \textit{random matching process} in which we (partially) reveal the partner $p(i)$ of each vertex $i \in [n]$ sequentially, starting from $i=1$ and proceeding in the order of the cycle. Formally, for some $i \geq 2$ suppose that we have run the process for vertices $1,2,\dots,i-1$; we partially reveal $p(i)$ via the following experiment:
\begin{itemize}
    \item Ask ``is $p(i) < i$?"
    \begin{enumerate}[label = (\alph*)]
        \item If the answer is ``no," then leave vertex $i$ \textit{unsaturated} and proceed to the next vertex.        
        \item If the answer is ``yes," then fully reveal the partner $p(i)$ by sampling it uniformly at random from the unsaturated vertices in $[i-1]$.
    \end{enumerate}
\end{itemize}

We stress that $p(i)$ is only fully revealed at time $i$ if we have $p(i) < i$. For $i \in [n]$, we let $\mathcal{F}_{i}$ be the $\sigma$-algebra generated by the history of the matching process through time $i$, and let $\mathcal{F}_{0}$ be the trivial $\sigma$-algebra corresponding to $M_{n}$. We define
$$
    X(i) = \left\lvert\{j \in [i]\,:\,p(j) > i \}\right\rvert
$$
for $i \in [n]$ and let $X(0) = 0$. In other words, $X(i)$ is the number of unsaturated vertices at time $i$ of the process. Note that $X(i)$ is $\mathcal{F}_{i}$-measurable. We make note of the following useful expression, which we use throughout our analysis:

\begin{lemma}\label{lem:backward_prob}
    For any $i \in [n]$ and $1\leq j < i$, we have
    $$
        \prob\Big(p(i) = j \,|\, \mathcal{F}_{i-1} \Big) = \frac{\unit\{j \textup{ is unsaturated}\}}{n-(i-1)}.
    $$
    Consequently,
    $$
        \prob \Big( p(i) < i \,|\,\mathcal{F}_{i-1} \Big) = \frac{X(i-1)}{n-(i-1)}.
    $$
\end{lemma}

\begin{proof}
    Suppose vertex $j \in [i-1]$ is unsaturated at time $i-1$. There are 
    \begin{equation}\label{eq:matching_exp1}
        (n-i)_{X(i-1)-1} \cdot M(n - i - (X(i-1)-1))
    \end{equation}
    ways to complete the matching which are consistent with $\mathcal{F}_{i-1}$ and the event $\{p(i) = j\}$, where $M(m)$ is the number of perfect matchings on $m$ elements and $(a)_b$ represents the falling factorial. (Explanation: we first match the remaining $X(i-1)-1$ unsaturated vertices in $[i-1]$ with partners in $[n] \setminus [i]$, then we match the remaining $n-i -(X(i-1)-1)$ vertices.) 
    
    On the other hand, there are 
    \begin{equation}\label{eq:matching_exp2}
        (n-(i-1))_{X(i-1)} \cdot M(n - (i-1) - X(i-1))
    \end{equation}
     valid ways in total to complete the matching consistent with $\mathcal{F}_{i-1}$. Dividing (\ref{eq:matching_exp1}) by (\ref{eq:matching_exp2}) gives the first expression in the statement. Summing this expression over $j \in [i-1]$ gives the second.
\end{proof}

Thus, the procedure of partially revealing the partner of vertex $i$ can be run as follows: given $\mathcal{F}_{i-1}$, flip a coin with probability $\frac{X(i-1)}{n-(i-1)}$ of being heads; on heads, choose a partner for $i$ uniformly at random from the $X(i-1)$ unsaturated vertices in $[i-1]$; on tails, leave $i$ unsaturated. Repeating this experiment independently for all $i \in [n]$ reveals the complete matching $M_{n}$ at time $n$ with the correct distribution. Since the execution of \textbf{Sudoku} at vertex $i$ depends only on knowing if $p(i) > i$ or on the colour assigned to $p(i)$ and $i-1$ if $p(i) < i$, we can couple a run of the algorithm with the matching process, which will be essential to our analysis. 

\subsection{Heuristic analysis and introduction to the differential equations method}\label{sec:heuristic}

Here we outline how the different pieces of our algorithm fit together and give some heuristic justification for Theorem~\ref{thm:main_theorem}.

We treat the first $i_{0} = \frac{n}{\log\log n}$ steps of the matching process as a \textit{burn-in} phase, during which we run a different algorithm that properly $3$-colours vertices $1$ through $i_{0}$ in a strongly balanced way. We include all of the $i_{0}$ initial vertices in the Sudoku set $S$, noting that they contribute only $o(n)$ to the total size. We then run \textbf{Sudoku}, started at vertex $i_{0}$ and with $S_{0} = [i_{0}]$, until $i_{1} = n - \frac{n}{\log\log n}$. The final $n-i_{1}$ vertices are processed in a \textit{completion phase}; per the discussion in Subsection~\ref{sec:greedy_alg}, we show that a.a.s.\ there is some sub-interval of vertices in $\{i_1 + 1, \ldots, n-2\}$ that induces an even cycle in $G_{n}$, which ensures that colouring can be completed. All vertices in the completion phase are included in $S$, which again adds $o(n)$ to the total size. The burn-in and completion phases are treated in Section~\ref{sec:burn_in}.

The main contribution to $|S|$ comes from vertices in $\{i_{0}, \dots, i_{1}\}$, where we run \textbf{Sudoku}. Recall the definitions of the sets $\mathcal{B}_{C}$, $\mathcal{B}_{U}^{(c)}$, and $\mathcal{B}_{U}^{(d)}$ from Subsection~\ref{sec:correctness}. Lemma~\ref{lem:size_bound} gives an upper bound on the size of $S(i_{1})$, the Sudoku set produced by \textbf{Sudoku} at time $i_{1}$, in terms of the sizes of these sets and $S_{0}$. For the purposes of the informal analysis, we will ignore $\mathcal{B}_{U}^{(c)}$ and $\mathcal{B}_{U}^{(d)}$, as very few vertices in $G_{n}$ will be included in these sets a.a.s. Since the contributions from the burn-in and completion phases will also be negligible, using Lemma~\ref{lem:size_bound} we have $|S| \leq (1+o(1))\frac{1}{2}|\mathcal{B}_{C}|.$

Now, vertex $i$ is in $\mathcal{B}_{C}$ if either i) it has a forward edge, or ii) it has a backward edge to a particular forbidden colour, that colour being determined by the type of vertex $i-1$. Importantly, these conditions do not (directly) reference the type of run that $i$ is a part of. Recall that $X(i)$ is the number of unsaturated vertices at time $i$ of the matching process. We additionally let $X_{k}(i)$ be the number of unsaturated vertices of colour $k$ at time $i$, so that $X(i) = X_{1}(i) + X_{2}(i) + X_{3}(i)$. Suppose that at time $i$ we have $X_{k}(i) \approx \frac{X(i)}{3}$ for each $k \in \{1,2,3\}$. Then we claim that the probability that vertex $i+1$ is in $\mathcal{B}_{C}$, conditioned on $\mathcal{F}_{i}$, is approximately
\begin{equation}\label{eq:bad_probability}
	\frac{X(i)}{3(n-i)} + 1-\frac{X(i)}{n-i} = 1-\frac{2X(i)}{3(n-i)},
\end{equation}
regardless of whether the algorithm is in a run of type $A$ or $B$, and regardless of the type of vertex~$i$. Indeed, by the balance assumption the number of unsaturated vertices in the forbidden colour class for vertex $i+1$ at time $i$ is approximately $\frac{X(i)}{3}$, and thus by Lemma~\ref{lem:backward_prob} the probability that $i+1$ is matched with one of them, conditioned on $\mathcal{F}_{i}$, is roughly $\frac{X(i)}{3(n-i)}$. Also by Lemma~\ref{lem:backward_prob}, the probability that $p(i+1) > i+1$ conditioned on $\mathcal{F}_{i}$ is exactly $1 - \frac{X(i)}{n-i}$. Together these give the left-hand side of (\ref{eq:bad_probability}).

Suppose that the balance condition $X_{k}(i) \approx \frac{X(i)}{3}$ holds for the duration of the algorithm. Also, suppose that there is some continuous function $x(t)$ on $[0,1]$ such that $X(i)$ stays tightly concentrated around $nx(i/n)$ for the duration of the process. (The differential equations method can be used to establish results of precisely this type.) Since 
$$
	\E[X(i+1) - X(i)\,|\,\mathcal{F}_{i}] = \left(1 - \frac{X(i)}{n-i} \right) - \frac{X(i)}{n-i} = 1 - \frac{2X(i)}{n-i}
$$
and $\frac{X(1)}{n-1} = \frac{1}{n-1}$ and $\frac{X(n)}{n} = 0$, it is natural to predict that the function $x(t)$ solves the boundary value problem
\begin{equation}\label{eq:de}
	x'(t) = 1 - \frac{2x(t)}{1-t},\quad x(0) = 0 \text{ and }x(1) =0
\end{equation}
which is solved by $x(t) = t(1-t)$. We remark here that we do not need the differential equations method to understand the trajectory of $X(i)$---this can be achieved, for instance, by applying concentration inequalities directly to the sums $\sum_{j=1}^{i}\unit\{p(j) > j\}$. We use differential equations only to understand the trajectories of the $X_{k}(i)$'s, which are more complicated. In doing so, we get $X(i)$ ``for free," since $X(i) = X_{1}(i) + X_{2}(i) + X_{3}(i).$

Using (\ref{eq:bad_probability}), we can approximately bound $|S|$ as follows:
\begin{equation}\label{eq:approximation}
	|S| \lesssim \frac{1}{2}|\mathcal{B}_{C} \approx \frac{n}{2}\int_{0}^{1} \left( 1-\frac{2x(t)}{3(1-t)} \right) \,dt = \frac{n}{2} \int_{0}^{1} \left( 1 - \frac{2t}{3} \right) \,dt = \frac{n}{3}.
\end{equation}
The main challenge in making the approximation (\ref{eq:approximation}) rigorous comes in establishing that the random variables $X_{k}(i)$ remain close to equal for the duration of the process. Applying the differential equations method directly to the $X_{k}(i)$'s is difficult, since $\E[X_{k}(i+1) - X_{k}(i)\,|\,\mathcal{F}_{i}]$, the expected one-step change for a particular $k \in \{1,2,3\}$, depends explicitly on the type of vertex $i$. 

To overcome this, we will analyze changes in the variables $X_{k}(i)$ over segments of length $\omega$ for $\omega=\omega(n)$ going to infinity slowly. The reasoning behind this is that over longer time intervals the influence of the type of any single vertex on the $X_{k}(i)$'s is negligible. On these segments, the trajectory of Sudoku can be approximated by a Markov chain which makes transitions between the vertex types in $\mathcal{V}$. One complication that arises in the approximation is the need for a sufficiently large stock of unsaturated vertices in order for the chain to mix quickly. This is the reason for the burn-in phase: it ensures that we accumulate many unsaturated vertices at the beginning of the process, and that the colours classes are balanced with respect to these vertices.

Formally, given a length $i_{0} =\frac{n}{\log\log n}$ for the burn-in phase and a length $\omega$ for the segments, we will define
\begin{eqnarray*}
	\wt{X}_{k}(i) &=& \frac{1}{\omega}X_{k}(i_{0} + i\omega) \quad \text{ for }k \in \{1,2,3\}\\
	\wt{X}(i) &=& \frac{1}{\omega}X(i_{0} + i\omega) = \wt{X}_{1}(i) + \wt{X}_{2}(i) + \wt{X}_{3}(i)\\
	\wt{\mathcal{F}}_{i} &=& \mathcal{F}_{i_{0}+i\omega}.
\end{eqnarray*}
Letting $N = \frac{n}{\omega}$ and $i_{1} = n-\frac{n}{\log \log n}$, we will show that for each $k \in \{1,2,3\}$ and $0 \leq i \leq \frac{i_{1} - i_{0}}{\omega} = N\left(1 - \frac{2}{\log \log n} \right)$,
\begin{equation}
	\E[\wt{X}_{k}(i+1) - \wt{X}_{k}(i)\,|\,\wt{\mathcal{F}}_{i}] = (1 + o(1))\frac{1}{3}\left(1 - \frac{\wt{X}(i)}{N-i} \right)
\end{equation}
which implies that $\wt{X}_{k}(i)$ is well-approximated by $\frac{N}{3}x\left(\frac{i_{0}}{n} + \frac{i}{N} \right)$ for all $k$. In particular, we expect the $\wt{X}_{k}(i)$'s (and so also the $X_k(i)$'s) to remain balanced for almost the entire process.

\section{The burn-in and completion phases}\label{sec:burn_in}

Here we establish some features of the burn-in and completion phases of our colouring algorithm. For the burn-in phase, we will run the matching process up to time $i_{0} = \frac{n}{\log \log n}$ and properly $3$-colour the corresponding vertices so that for any two colours $k$ and $\ell$, the number of unsaturated vertices of colour $k$ differs from the number of unsaturated vertices of colour $\ell$ by at most $1$. For the completion phase, we will show that some sub-interval of vertices in $\left\{i_{1} + 1, \dots, n \right\}$ induces an even cycle in $G_{n}$ a.a.s.

\subsection{The burn-in phase}

Our goal is to colour the vertices in $\{1,2,\dots,i_{0}\}$ in such a way that the colour classes remain as balanced as possible with respect to unsaturated vertices. To achieve this, we use a variant of the greedy colouring algorithm which we call \textbf{BalancedGreedy}.
 
In \textbf{BalancedGreedy} we divide the interval $\{1,2,\dots,i_{0}\}$ into $\frac{i_{0}}{7}$ batches of $7$ consecutive vertices apiece, and possibly one extra batch at the end which can be shorter. We say that vertex $i \in [i_{0}]$ has a \emph{long forward edge} if $p(i)>i_0$, and call a batch \textit{good} if it has $7$ vertices with long forward edges. All other batches are bad. (This means the extra batch at the end is always bad if it exists.)

We begin by partially revealing the edge of each vertex in $[i_{0}]$ and greedily colouring the bad batches in  order. Next we colour the good batches. Say that there are $L$ good batches in total, and for $j = 0,1,\dots,L$ we let $X_{k}^{(j)}(i_{0})$ be the number of unsaturated vertices in $[i_{0}]$ of colour $k$ just after the $j$th good batch is coloured, or just after the last bad batch is coloured in the case $j = 0$. Let $D(j) = \max_{k}X_{k}^{(j)}(i_{0}) - \min_{k}X_{k}^{(j)}(i_{0}).$ Observe that $D(L) = \max_{k}X_{k}(i_{0}) - \min_{k}X_{k}(i_{0})$, since after colouring the $L$th good batch all vertices in $[i_{0}]$ have been coloured.

Now, it is not hard to show that in any good batch, vertices $2,4$ and $6$ can be coloured using the same colour, and vertices $1,3,5$ and $7$ can be coloured using each of the remaining two colours twice. Further, it is also always possible to use an arbitrary colour only once in a good batch and use the other two colours three times each. These properties ensure that we always have $D(j) \leq D(j-1)-1$ if $D(j-1) > 1$, or that $D(j) \in \{0,1\}$ if $D(j-1) \in \{0,1\}$. If the number of good batches $L$ is large enough, we will be able to conclude that $D(L) = \max_{k}X_{k}(i_{0}) - \min_{k}X_{k}(i_{0}) \in \{0,1\}$, and indeed we show that this is the case in the next lemma.

\begin{lemma}\label{lem:burn_in2}
	A.a.s., \textup{\textbf{BalancedGreedy}} colours the vertices in $[i_{0}]$ so that 
    $$
        |X_{k}(i_{0}) - X_{\ell}(i_{0})| \leq 1
    $$ 
    for all $k,\ell \in \{1,2,3\}$.
\end{lemma}

\begin{proof}
    Each time we colour a bad batch, the colour discrepancy for unsaturated vertices trivially increases by at most $7$. Thus, if $L$ is the number of good batches,      
    $$
        D(0) \leq 7\cdot \#\text{bad batches} \leq 7\left((1+o(1))\frac{i_{0}}{7} - L\right) = (1+o(1))i_{0}  - 7L.
    $$
    Now, since each time we colour a good batch $D(\cdot)$ either decreases by at least $1$ or remains in $\{0,1\}$, we have
    $$
        D(L) \leq \max\{D(0) - L, 1\} \leq \max\left\{(1+o(1))i_{0}  - 8L, 1\right\}.
    $$
    In particular, if $L \geq \frac{i_{0}}{7.5}$, then $D(L) \leq 1$. 
    
    To see that this holds a.a.s., first observe that there are $o(i_{0})$ edges which join pairs of vertices in $[i_{0}]$ a.a.s.\ This follows from an easy first moment argument: we expect $O\left( i_0 \cdot \frac {i_0}{n} \right)= O\left(\frac{i_{0}}{\log\log n}\right)$ such edges, and thus by Markov's inequality there are $o(i_0)$ of them a.a.s. Since each such edge makes at most $2$ batches go bad, the number of good batches is at least $(1+o(1))\frac{i_{0}}{7} - o(i_0) \ge \frac {i_0}{7.5}$ a.a.s.
\end{proof}

We finish this section by showing that the number of unsaturated vertices at time $i_{0}$ is well-concentrated around what we expect. Recall from Section~\ref{sec:heuristic} that we define the function $x(t)$ on $(0,1)$ as 
$$
    x(t) = t(1-t).
$$

\begin{lemma}\label{lem:burn_in3}
	Let $\lambda =\lambda(n) \to \infty$ as $n \to \infty$ arbitrarily slowly. Then, a.a.s.
		$$
			\left\vert X(i_{0}) - nx(i_{0}/n)\right\vert = O(\sqrt{\lambda i_{0}}).
		$$
\end{lemma}

\begin{proof}
    For $j \in [i_{0}]$, let $Y_{j} = \unit\{p(j) > j\}$ and let $Y = \sum_{j=1}^{i_{0}}Y_{j}$. We may write
    $$
        X(i_{0}) = \sum_{j=1}^{i_{0}}\unit\{p(j) > j\} - \sum_{j=1}^{i_{0}}\unit\{p(j) < j\} = 2Y - i_{0}.
    $$
    Since $\prob(p(j) > j) = \frac{n-j}{n-1}$, we get 
    $$
        \E[Y] = \sum_{j=1}^{i_{0}}\frac{n-j}{n-1} = i_{0} - \frac{1}{n-1}\sum_{j=1}^{i_{0}}(j-1) = i_{0} - \frac{(i_{0}-1)i_{0}}{2(n-1)} = i_{0}\left(1 - \frac{i_{0}}{2n} \right) + O(1)
    $$
    from which we easily derive that $\E[X(i_{0})] = nx(i_{0}/n) + O(1).$ To show the result, it will suffice to show that $Y$ is within $O(\sqrt{\lambda i_{0}})$ of $\E[Y]$ a.a.s.\ for any $\lambda \to \infty$. To do this, we will show that both collections of indicators $\{Y_{j}\}_{j=1}^{i_{0}+1}$ and $\{1-Y_{j}\}_{j=1}^{i_{0}+1}$ are negatively correlated. The desired concentration is then implied by Proposition~\ref{prop:chernoff}. (As an aside, it is perhaps a more natural idea to consider the sum $\sum_{j=1}^{i_{0}}\unit\{p(j) > i_{0}\}$ directly. However, the collection $\{\unit\{p(j) > i_{0}\}\}_{j=1}^{i_{0}}$ is actually \textit{positively} correlated, which prevents us from applying large deviations bounds to their sum.)

    Let $1 \leq j_{1} < j_{2} < \cdots < j_{k} \leq i_{0}+1$. We claim
    \begin{equation}\label{eq:negative_cor_prod}
        \prob(Y_{j_{h}}=1 \,\forall\,h = 1,2,\dots,k) = \frac{n-j_{k}}{n-1} \cdot \frac{n-j_{k-1}-2}{n-3}\cdots \frac{n-j_{1} - 2(k-1)}{n-(2(k-1)+1)}.
    \end{equation}
    Indeed, by assigning neighbours for $j_{1},\dots,j_{k}$ in decreasing order, we see that there are $n-j_{k}$ choices for the neighbour of $j_{k}$, then $n-j_{k-1}-2$ choices for the neighbour of $j_{k-1}$, etc., which are consistent with $p(j_{h}) > j_{h}$ for $h = 1,\dots,k.$ Moreover, any particular assignment of neighbours for these vertices occurs with probability $\prod_{h=0}^{k-1}\frac{1}{n-(2h-1)}$, which gives (\ref{eq:negative_cor_prod}).

    Observe that for any $h = 1,\dots,k$, the factor corresponding to $h$ in (\ref{eq:negative_cor_prod}) is
    $$
        \frac{n-j_{h}-2(k-h)}{n-(2(k-h)+1)} \leq \frac{n-j_{h}}{n-1} = \prob(Y_{j_{h}}=1).
    $$
    From the above, we immediately get that the product in (\ref{eq:negative_cor_prod}) is upper bounded by $\prod_{h=1}^{k}\prob(Y_{j_{h}}=1)$. We conclude that the collection $\{Y_{j}\}_{j=1}^{i_{0}+1}$ is negatively correlated. Similar computations give that $\{1-Y_{j}\}_{j=1}^{i_{0}+1}$ is also negatively correlated, completing the proof.
\end{proof}

\subsection{The completion phase}

 The main result of this subsection is the following lemma. Before stating it, we recall that $i_0=\frac{n}{\log\log n}$ and $i_1= n-\frac{n}{\log\log n}$.

\begin{lemma}\label{lem:completion}
	A.a.s., there is a sequence of consecutive vertices in $\{i_1 + 1, \dots, n\}$ which induce an even cycle in $G_{n}$. 
\end{lemma}

\begin{proof}[Proof of Lemma \ref{lem:completion}]
	For notational simplicity, we will prove instead that there is a sequence of consecutive vertices in $\{1,2,\dots,i_0\}$ which an induce an even cycle. By symmetry, the same clearly holds when $\{1,2,\dots,i_0\}$ is replaced with $\{i_1 + 1, \dots,n\}$.
	
	Divide the interval $\{1,2,\dots,i_0\}$ into $\sqrt{n}/{\log \log n}$ consecutive intervals of length $\sqrt{n}$, and label these intervals $I_{1}, I_{2}, \dots, I_{\sqrt{n}/\log \log n}$.  We will show that a.a.s., at least one interval $I_{j}$ satisfies the following:    
    \begin{enumerate}[label = (\roman*)]
        \item there is a single edge in $M_{n}$ matching a pair of vertices in $I_{j}$, and
        \item the pair of vertices matched by this edge are at odd distance from one another on the cycle.
    \end{enumerate}
    Note that (i) guarantees $I_{j}$ contains an induced cycle, and (ii) guarantees the cycle is even. 

	We will partially reveal edges on the vertices in order until we find some $I_{j}$ satisfying (i) and (ii) above, or until we have processed all vertices in $\bigcup_{j}I_{j}$. Suppose we are about to process the vertices in $I_{j}$ for some $j \geq 1$. We claim that, conditioned on the history $\mathcal{F}_{(j-1)\sqrt{n}}$ up to this point (recall $\mathcal{F}_{0} = \varnothing$ in the case $j = 1$), the probability that $I_{j}$ satisfies (i) and (ii) is at least 	
	\begin{equation}\label{eq:ij_success}
		(1+o(1)) \frac{1}{2n} \binom{\sqrt{n}}{2} \left(1 - (1+o(1))\frac{1}{\sqrt{n}} \right)^{\sqrt{n}} = (1+o(1))\frac{1}{4e}.
	\end{equation}
	
	To prove (\ref{eq:ij_success}), consider any vertex $i \in I_{j}$. Given $\mathcal{F}_{i-1}$, the probability that $p(i) \not\in \{(j-1)\sqrt{n} +1,\dots, i-1\}$ is at least 
	\begin{equation}\label{eq:conditional_bound}
		1 - \frac{i - 1 - (j-1)\sqrt{n} }{n-(i-1)} \geq 1 - (1+o(1))\frac{1}{\sqrt{n}}.
	\end{equation}
		Indeed, the first expression above holds because once the process reaches vertex $i$, in the ``worst case" all vertices in $\{(j-1)\sqrt{n} + 1, \dots, i-1\}$ are unsaturated, and $i$ is matched with any one of these vertices with probability $\frac{1}{n-(i-1)}$ by Lemma~\ref{lem:backward_prob}. For the second inequality, we use that $i \leq j \sqrt{n} $. 
	
	Now, (\ref{eq:ij_success}) easily follows using the bound (\ref{eq:conditional_bound}). Indeed, there are $(1+o(1))\frac{1}{2}\binom{ \sqrt{n}}{2}$ pairs $\{l, r\} \subset I_{j}$ with $l < r$ and $r - l$ odd. For any given such pair, the probability that $r$ and $l$ are matched by $M_{n}$ is $(1+o(1))\frac{1}{n}$; the probability that no other vertex in $I_{j}$ creates a cycle in $I_{j}$ when its edge is revealed is $\left(1 - (1+o(1))\frac{1}{\sqrt{n}} \right)^{\sqrt{n}}$. Combining all of these observations yields (\ref{eq:ij_success}).
	
	The bound (\ref{eq:ij_success}) gives
	$$
		\prob\left(I_{j} \text{ has no induced even cycle}\,|\,\mathcal{F}_{(j-1)\sqrt{n}} \right) \leq 1 - (1+o(1))\frac{1}{4e}.
	$$
	With the tower property of conditional expectation, it follows that 
	$$
		\prob\left(\bigcap_{j=1}^{\sqrt{n}/\log \log n}\{I_{j} \text{ has no induced even cycle}\}\right) \leq \left(1 -(1 + o(1))\frac{1}{4e}\right)^{\sqrt{n}/\log\log n} = o(1)
	$$
	which completes the proof.	
\end{proof}

\section{The Markov chain on types}\label{sec:markov_chain}

\subsection{Definition and basic properties of the chain}

Here we define a Markov chain $Q$ on the set of types $\mathcal{V} = \mathcal{B} \cup \mathcal{A}$ which is used in the computation of $\E[\wt{X}_{k}(i+1) - \wt{X}_{k}(i)\,|\,\wt{\mathcal{F}}_{i}]$. Throughout, we identify a Markov chain with its transition matrix, so we think of $Q$ as an $18 \times 18$ matrix with rows and columns indexed by $\mathcal{V}.$ The chain $Q$ is meant to reflect an idealized run of a segment of \textbf{Sudoku} of length $\omega$, starting from vertex $i$ and assuming that vertices in $[i]$ have already been processed. When designing $Q$, we make a couple of heuristic assumptions on the matching process which are technically false, but not by much. They are:
	\begin{enumerate}[label = (\roman*)]
		\item for any $i < j \leq i+\omega$, the probability that $j$ is joined to any particular unsaturated vertex is $\frac{1}{n-i}$;
		\item for any $i < j \leq i+\omega$, we have $p(j) \not\in \{i+1, i+2, \dots, j-1\}$.
	\end{enumerate} 
Item (i) says that we ``freeze" the matching probabilities at time $i$. (In reality, the correct probability at time $j$ is $\frac{1}{n-j}$, by Lemma~\ref{lem:backward_prob}.) Item (ii) says that we ignore edges that match two vertices $\{i+1, \dots, i+ \omega\}$. In practice, we will consider segments of length $\omega = o(\sqrt{n})$; for any such segment, the probability that we encounter one of these ``short" edges is $o(1)$, meaning that assumption (ii) is not overly restrictive.

$Q$ takes as parameters the ($\mathcal{F}_{i}$-measurable) random variables $X_{k}(i)$ for $k \in \{1,2,3\}.$ To illustrate the design principle behind $Q$, suppose that at time $j \geq i$ the algorithm is in state $A_{b}^{(1)}$. When revealing the partner of $j+1$, we get a good edge if $p(j+1) < j+1$ and $c(p(j+1)) = 2$ or $c(p(j+1)) = 3$, corresponding to transitions to states $A_{b}^{(3)}$ and $A_{b}^{(2)}$, respectively. Given $\mathcal{F}_{j}$, these events occur with probabilities $\frac{X_{2}(j)}{n-j}$ and $\frac{X_{3}(j)}{n-j}$, respectively. If the Markov chain is in state $A_{b}^{(1)}$ at time $j \geq i$, a transition to state $A_{b}^{(k)}$ for $k = 2$ or $k=3$ will occur with probability $\frac{X_{k}(i)}{n-i}$. 

Other transition probabilities can be deduced similarly. We make one further exception when considering states in $\mathcal{B}$; with these states, there is a possibility that vertex $j+1$ is joined to an unsaturated vertex which is exactly one vertex ahead of the pointer, which puts us in case ($B2$c) or ($B2$d) and causes a more complicated (though rare) colour selection and pointer update in \textbf{Sudoku}. The chain $Q$ simply ignores these transitions. 

We express the transition probabilities of $Q$ in terms of general parameters $q_{1}, q_{2}, q_{3} \in [0,1]$ satisfying $q =q_{1} + q_{2} + q_{3} \leq 1$. (In the actual analysis, we will set $q_{k} = \frac{X_{k}(i)}{n-i}$.) In the following expressions, we assume $\{k,\ell,m\} = \{1,2,3\}$. For $e \in \{b,f\}$,
\begin{align*}
	Q(A_{e}^{(k)}, A_{b}^{(\ell)}) &= q_{m}\\
        Q(A_{e}^{(k)}, A_{b}^{(m)}) &= q_{\ell}\\
	Q(A_{e}^{(k)}, B_{f}^{(k\ell)}) =Q(A_{e}^{(k)}, B_{f}^{(km)})&= \frac{1-q}{2}\\
	Q(A_{e}^{(k)}, B_{b}^{(k\ell)}) = Q(A_{e}^{(k)}, B_{b}^{(km)})&= \frac{q_{k}}{2}\\
	Q(A_{e}^{(k)}, V) &=0 \qquad\qquad \text{ for all other }V \in \mathcal{V}\\
    \intertext{and}
	Q(B_{e}^{(k\ell)}, A_{b}^{(k)}) &= q_{m}\\
	Q(B_{e}^{(k\ell)}, A_{f}^{(m)}) &= 1-q\\
	Q(B_{e}^{(k\ell)}, B_{b}^{(km)}) &= q_{k}\\
	Q(B_{e}^{(k\ell)}, B_{b}^{(\ell m)}) &= q_{\ell}\\
	Q(B_{e}^{(k\ell)}, V) &= 0 \qquad\qquad \text{ for all other }V \in \mathcal{V}.
\end{align*}
The transition probabilities above determine a Markov chain on $\mathcal{V}$ for any values of $q_{1},q_{2},q_{3} \in [0,1]$. When $q_{1},q_{2},q_{3} \in (0,1)$, $Q$ is aperiodic and irreducible, and thus it has a unique stationary distribution $\pi$ (dependent on the $q_{k}$'s). To verify irreducibility, we created a Jupyter notebook that constructs a graph based on the states of $Q$ and confirms that it is strongly connected. The notebook is available at the GitHub repository\footnote{\url{https://github.com/dwillhalm/sudoku_number/blob/main/irreducibility.ipynb}}. The irreducible chain $Q$ is aperiodic since $$Q^2(A_b^1,A_b^1) \ge Q(A_b^1,A_b^2)Q(A_b^2,A_b^1) = q_3^2 > 0$$ and $$Q^3(A_b^1,A_b^1) \ge Q(A_b^1,A_b^2)Q(A_b^2,A_b^3)Q(A_b^3,A_b^1) = q_3 q_1 q_2 > 0.$$
In the balanced regime, i.e., when $q_{k} = \frac{q}{3}$ for all $k$, we use the notation $Q_{\text{bal}}$ for the transition matrix. We have an explicit expression for the stationary distribution of $Q_{\text{bal}}$, which we denote by $\pi_{\text{bal}}$:
	\begin{equation}\label{eq:stationary_dist}
		\pi_{\text{bal}}(V) = \begin{cases} 
				\frac{q}{6}  & \quad\text{ if }V \in \mathcal{A}_{b}\\
				\frac{q}{12} & \quad\text{ if }V \in \mathcal{B}_{b}\\
				\frac{1-q}{6}  & \quad\text{ if }V \in \mathcal{A}_{f}\\
				\frac{1-q}{12} & \quad\text{ if }V \in \mathcal{B}_{f}
			\end{cases}.
	\end{equation}
Here $\mathcal{B}_{b}$ is the set of types in $\mathcal{B}$ with a backward edge, $\mathcal{A}_{b}$ the set of types in $\mathcal{A}$ with a backward edge, etc. It is straightforward to verify that $\pi_{\text{bal}}$ is indeed stationary for $Q$ in the balanced regime. 

Before analyzing the chain, we recall some definitions and basic results about matrix and vector norms as well as Markov chain mixing. For a vector $v \in \R^{K}$, the $\ell^{1}$ and $\ell^{\infty}$ norms of $v$ are, respectively,
$$
    \norm{v}_{1} = \sum_{i=1}^{K}|v_{i}| \quad \text{ and } \quad \norm{v}_{\infty} = \max_{i \in [K]}|v_{i}|.
$$
Any norm $\norm{\cdot}$ on $\R^{K}$ induces a norm on matrices $M \in \R^{K \times K}$ via
$$
    \norm{M} = \sup_{\substack{v \in \R^{K} \\ \norm{v} = 1}} \norm{Mv}.
$$
It is an easy exercise to show that the following expressions hold for the $\ell^{1}$ and $\ell^{\infty}$ matrix norms:
\begin{equation}\label{eq:norm_identities}
    \norm{M}_{1} = \max_{j \in [K]}\sum_{i=1}^{K}|M_{ij}| \quad \text{ and } \quad \norm{M}_{\infty} = \max_{i \in [K]}\sum_{j=1}^{K}|M_{ij}|.
\end{equation}
In the study of Markov chain mixing, it is conventional to consider convergence in the total-variation norm $\norm{\cdot}_{TV}$. For probability distributions $\mu,\nu \in \R^{K}$, define
$$
    \norm{\mu - \nu}_{TV} = \max_{A \subseteq [K]} \Big|\sum_{i\in A} \mu_i - \nu_i\Big|.
$$
We note that $\norm{\mu - \nu}_{TV} = \frac{1}{2}\norm{\mu - \nu}_{1}$ (see \cite[Proposition 4.2]{levin2017markov}). 

Let $P$ be the transition matrix for an aperiodic, irreducible Markov chain on a state space of size $K$, and let $\pi$ be its stationary distribution. Let $d(t) = \sup_{\mu}\norm{\mu P^{t} - \pi}_{TV}$, where the $\sup$ is taken over all probability distributions $\mu$ on the state space of $P$. The $\vep$-mixing time of $P$, denoted $t_{\text{mix}}(\vep)$, is defined as
$$
    t_{\text{mix}}(\vep) = \min\{t\,:\,d(t) \leq \vep\}.
$$

\subsection{Perturbations of the balanced regime}

Our goal is to understand the behaviour of the chain $Q$ when  $(1 - \gamma)\frac{q}{3} < q_{k} < (1 + \gamma)\frac{q}{3}$ holds for all $k$ for some $\gamma \in (0,1)$, which is a perturbation of the balanced regime. Specifically, we approximate the stationary distribution and bound the mixing time of $Q$ under such a perturbation. To do this we make repeated use of Lemma~\ref{lem:transitions} below. 

\begin{lemma}\label{lem:transitions}
	There exists a constant $C_{\textup{cond}} > 1$ so that the following holds for all $q \in (0,1)$: let $(W_{i})_{i\in \N_{0}}$ be a Markov chain on $\mathcal{V}$ with transition matrix $Q_{\textup{bal}}$ with parameter $q$; for any $V \in \mathcal{V}$ and any $V' \in \mathcal{A}_{b} \cup \mathcal{B}_{b}$, we have
	$$
		\prob_{V}(W_{6} = V') \geq \frac{\pi_{\textup{bal}}(V')}{C_{\textup{cond}}}
	$$
	and for any $V \in \mathcal{V}$ and $V' \in \mathcal{A}_{f} \cup \mathcal{B}_{f}$, 
	$$
		\prob_{V}(W_{i} = V' \textup{ for some }i \in \{0,1,\dots,6\}) \geq \frac{\pi_{\textup{bal}}(V')}{C_{\textup{cond}}},
	$$
    where $\prob_V$ denotes probability conditioned on the chain started at $V$.
\end{lemma}

The proof is computationally intensive; therefore, we verified it using a Jupyter notebook that computes powers of $Q_\textup{bal}$ symbolically with the SimPy library. The notebook is available at the GitHub repository\footnote{\url{https://github.com/dwillhalm/sudoku_number/blob/main/matrix_calcs.ipynb}}. The constant $C_{\text{cond}}$ in Lemma~\ref{lem:transitions} is so-named because of its relationship to a \textit{condition number} of the chain $Q_{\text{bal}}$. This number measures the sensitivity of the stationary distribution $\pi_{\text{bal}}$ to perturbations of the parameters $q_{1}, q_{2}$, and $q_{3}$ away from the balanced regime. We formalize this in the next lemma. 

\begin{lemma}\label{lem:condition_number}
    Let $q \in (0,1)$ and $q = q_{1} + q_{2} + q_{3}$. Let $\gamma \in (0,1)$ and suppose $(1-\gamma)\frac{q}{3} < q_{k} < (1 + \gamma)\frac{q}{3}$ for all $k$. Let $Q$ be the corresponding transition matrix. Then the stationary distribution $\pi$ of $Q$ satisfies	
	$$
		\norm{\pi - \pi_{\textup{bal}}}_{\infty} \leq 3C_{\textup{cond}}\gamma q
	$$
    where $\pi_{\textup{bal}}$ is the stationary distribution of the chain $Q_{\textup{bal}}$ with parameter $q$.
\end{lemma}

\begin{proof}
	We use a result from~\cite{cho2000markov}: for irreducible transition matrices $P$ and $\wt{P} = P - E$ on a finite state space $\mathcal{X}$ with stationary distributions $\pi_{P}$ and $\pi_{\wt{P}}$, respectively, we have
	\begin{equation}\label{eq:sensitivity}
		\norm{\pi_{P} - \pi_{\wt{P}}}_{\infty} \leq \frac{1}{2}\kappa(P) \norm{E}_{\infty}
	\end{equation}
	where
	$$
		\kappa(P) = \max_{y \in \mathcal{X}}\max_{x \neq y}\pi_P(y)\E_{x}[\tau_{y}].
	$$
	Here, $\tau_{y}$ denotes the hitting time of $y$ of the chain $P$ and $\E_{x}$ denotes expectation conditioned on the chain $P$ started at $x$. Let $E = Q_{\text{bal}} - Q$. Using the transition probabilities for $Q_{\text{bal}}$ and $Q$, we observe that $\norm{E}_{\infty} \leq \gamma q$. 
	
	Let $(W_{i})_{i \in \N_{0}}$ be a walk according to $Q_{\text{bal}}$ and $\tau_V$ be the hitting time of state $V$ by the walk $(W_i)_{i\in\N_0}$. Lemma~\ref{lem:transitions} implies that for any state $V' \in \mathcal{V}$, 
	\begin{equation}\label{eq:hitting_prob}
		\prob_{V}(W_{i} = V'\text{ for some }i=0,1,\dots,6) \geq \frac{\pi_{\text{bal}}(V')}{C_{\text{cond}}} \quad \text{ for all }V \in \mathcal{V}
	\end{equation}
	from which it follows that $\prob_{V}(\tau_{V'} > 6i) \leq (1- \frac{\pi_{\text{bal}}(V')}{C_{\text{cond}}})^{i}$ for all $i \in \N_{0}$ and any $V.$ Thus	
	\begin{align*}
		\E_{V}[\tau_{V'}] &= \sum_{i = 1}^{\infty}\prob_{V}(\tau_{V'} \geq i)
		\leq 6\sum_{i=0}^{\infty}\prob_{V}(\tau_{V'} > 6i)
		\leq 6 \sum_{i=0}^{\infty}\left(1- \frac{\pi_{\text{bal}}(V')}{C_{\text{cond}}}\right)^{i}
		=\frac{6C_{\text{cond}}}{\pi_{\text{bal}}(V')}
	\end{align*}
	where for the first inequality we use that $\prob_{V}(\tau_{V'} \geq i)$ is non-increasing in $i$. In particular, for any $V, V'$ we have $\pi_{\text{bal}}(V')\E_{V}[\tau_{V'}] \leq 6C_{\text{cond}}$, which implies $\kappa(Q_{\text{bal}}) \leq 6  C_{\text{cond}}$.
\end{proof}

Lemma~\ref{lem:transitions} also allows us to bound the mixing time of a perturbation of $Q_{\text{bal}}$ provided that the perturbation parameter $\gamma$ is small enough. In the proof, it will be clear that the result holds provided $\gamma$ is smaller than some absolute constant (i.e., not dependent on $q$). Since in our application $\gamma$ will  be $o(1)$, we elect not to make this constant explicit for the sake of simplicity.

\begin{lemma}\label{lem:mixing_time}
		Let $q \in (0,1)$ and $q_{1},q_{2},q_{3}$ satisfy $q = q_{1} + q_{2} + q_{3}$ and $(1- \gamma)\frac{q}{3} < q_{k} < (1+\gamma)\frac{q}{3}$ for some $\gamma \in (0,1)$ sufficiently small. Let $Q$ be the corresponding transition matrix. The mixing time of $Q$ satisfies
		$$
			t_{\textup{mix}}(\vep) = O\left(\frac{1}{q}\log \left(\frac{1}{\vep} \right) \right).
		$$
\end{lemma}

\begin{proof}
    We will use a \textit{minorization condition} on $Q$ to derive the bound on the mixing time. Suppose we show that for some $i_{0} \in \N$ the following inequality holds:
    \begin{equation}\label{eq:alpha}
        \alpha := \sum_{V' \in \mathcal{V}}\min_{V \in \mathcal{V}}Q^{i_{0}}(V, V') > 0.
    \end{equation}
   Then $Q$ satisfies the minorization condition $Q^{i_{0}}(V, \cdot) \geq \alpha \nu(\cdot)$, where $\nu$ is the probability distribution on $\mathcal{V}$ defined by 
    $$
        \nu(V') = \frac{\min_{V \in \mathcal{V}}Q^{i_{0}}(V,V')}{\alpha}.
    $$
    By \cite[Proposition 2]{rosenthal1995minorization}, this implies $\norm{\mu Q^{i} - \pi}_{TV} \leq (1 - \alpha)^{\lfloor i/i_{0} \rfloor}$ for any $i \in \N_{0}$ and any initial distribution $\mu$ on $\mathcal{V}$. We will show that \eqref{eq:alpha} holds for $i_{0} = 6$ with $\alpha$ at least to small constant times $q$, from which we derive $t_{\text{mix}}(\epsilon) = O\left(\frac{1}{q}\log\left(\frac{1}{\epsilon} \right)\right)$.
    
	Let $Q_{\text{bal}}$ be the balanced transition matrix with parameter $q$ and let $E = Q_{\text{bal}} - Q$. As in the proof of Lemma~\ref{lem:condition_number}, we have that $\norm{E}_{\infty} \leq \gamma q$. It is an easy exercise to show that
	$$
		\norm{Q_{\text{bal}}^{6} - Q^{6}}_{\infty} \leq \sum_{m=1}^{6}\binom{6}{m}(\gamma q)^{m} = O(\gamma q).
	$$	
	By the above and Lemma~\ref{lem:transitions}, we have
	\begin{equation}\label{eq:transition_prob_lb} \alpha \geq
		 \min_{V \in \mathcal{V}}Q^{6}(V, A_{b}^{(1)}) \geq \min_{V \in \mathcal{V}}Q_{\text{bal}}^{6}(V, A_{b}^{(1)}) - O(\gamma q) \geq q\left(\frac{1}{C_{\text{cond}}} - O(\gamma) \right).
	\end{equation}
	Observe that the right-hand side of (\ref{eq:transition_prob_lb}) is at least a small constant times $q$ provided that $\gamma$ is sufficiently small with respect to $\frac{1}{C_{\text{cond}}}.$ 
\end{proof}

\subsection{Approximation by $Q$}

Here we compute the error incurred by approximating a short segment of \textbf{Sudoku} with a run of the chain $Q.$

\begin{lemma}\label{lem:comparison}
	Let $\omega = \omega(n) \in \N$ and let $i \in [n]$ be such that $\frac{\omega}{n-i} = o(1)$. Let $Q$ be the transition matrix with parameters $q_{k} = \frac{X_{k}(i)}{n-i}$ for $k \in \{1,2,3\}$. Let $V_{j} \in \mathcal{V}$ be the type given to vertex $j$ in \textbf{Sudoku}. For $j = 0,1, \dots, \omega$, let $\rho_{j}$ be the probability distribution on $\mathcal{V}$ given by
	$$
		\rho_{j}(V) = \prob(V_{i+j} = V\,|\,\mathcal{F}_{i}).
	$$
	Then, uniformly in $j = 0,1,\dots,\omega$,
	$$
		\norm{\rho_{j} - \rho_{0}Q^{j}}_{1} = O\left( \frac{\omega^{2}}{n-i}\right).
	$$
\end{lemma}

\begin{proof}
	For $j = 0, 1,\dots, \omega$ and $k \in \{1,2,3\}$, define 
	
	$$
		q_{k}(j)= \frac{X_{k}(i+j)}{n-(i+j)}
	$$
	and let $Q_{j}$ be the transition matrix with parameters $q_{1}(j), q_{2}(j), q_{3}(j)$. The distribution $Q_{j}(V_{i+j},\cdot)$ is quite similar to the distribution $\prob(V_{i+j+1} = \cdot\,|\,\mathcal{F}_{i+j}),$ which is the true one-step transition distribution on types given $\mathcal{F}_{i+j}$. Indeed, these distributions are exactly equal but for one rare case: if vertex $i+j$ is part of a run of type $B$ whose first vertex $i+j'$ (with $j' < j$) is unsaturated at time $i+j$, then with probability $\frac{1}{n-(i+j)}$, vertex $i+j+1$ is joined to $i+j'$ when the next edge in the graph is revealed. This results in a transition from $V_{i+j}$ to $V_{i+j+1}$ which is (potentially) not allowed under $Q_{j}$. However, even when this is the case, at most two coordinates of $Q_{j}(V_{i+j}, \cdot)$ and $\prob(V_{i+j+1} = \cdot\,|\,\mathcal{F}_{i+j})$ differ, each by at most $\frac{1}{n-(i+j)} \leq \frac{1}{n-(i+\omega)}$ in absolute value. Thus, we always have the bound	
	\begin{equation}\label{eq:tv_bound1}
		\norm{Q_{j}(V_{i+j}, \cdot) - \prob(V_{i+j+1} = \cdot\,|\,\mathcal{F}_{i+j})}_{1} \leq \frac{2}{n-(i+j)} \leq \frac{2}{n-(i+\omega)} = O\left(\frac{1}{n-i} \right).
	\end{equation}
	We can also easily bound $\norm{Q(V_{i+j}, \cdot) - Q_{j}(V_{i+j}, \cdot)}_{1}$. For any $k$, the value of $X_{k}(\cdot)$ changes by at most $j$ over the interval $[i+1,\dots,i+j]$. Thus, 
	\begin{equation}\label{eq:j_transition_lb}
		q_{k}(j) = \frac{X_{k}(i+j)}{n-(i+j)} \geq \frac{X_{k}(i) - j}{n-i} \geq q_{k} - \frac{\omega}{n-i}
	\end{equation}
	and
	\begin{eqnarray}\label{eq:j_transition_ub}
		q_{k}(j) = \frac{X_{k}(i+j)}{n-(i+j)} &\leq& \frac{X_{k}(i) + \omega}{n-(i+\omega)} \nonumber \\
		&=&q_{k}\left( \frac{1}{1-\frac{\omega}{n-i}}\right) + \frac{\omega}{n-(i+\omega)} \nonumber \\
		&=& q_k \left(1 + O\left( \frac{\omega}{n-i}\right)\right) + O\left( \frac{\omega}{n-i}\right) = q_{k} + O\left( \frac{\omega}{n-i}\right).
	\end{eqnarray}
	 Together, (\ref{eq:j_transition_lb}) and (\ref{eq:j_transition_ub}) give that $|q_{k} - q_{k}(j)| = O\left( \frac{\omega}{n-i}\right)$ for all $k$. Note that this also implies $|q-q(j)| = O\left(\frac{\omega}{n-i} \right)$, and it easily follows that 
	 \begin{equation}\label{eq:tv_bound2}
	 	\norm{Q(V_{i+j}, \cdot) - Q_{j}(V_{i+j}, \cdot)}_{1} = O\left( \frac{\omega}{n-i} \right).
	 \end{equation}	
	Applying the triangle inequality and using the bounds (\ref{eq:tv_bound1}) and (\ref{eq:tv_bound2}) yields
	
	\begin{equation}\label{eq:tv_bound3}
		\norm{Q(V_{i+j}, \cdot) - \prob(V_{i+j+1} = \cdot\,|\,\mathcal{F}_{i+j})}_{1} = O\left( \frac{\omega}{n-i} \right).
	\end{equation}
	Now, let $\rho_{j}$ be the distribution $\prob(V_{i+j} = \cdot\,|\,\mathcal{F}_{i})$. (We think of $\rho_{j}$ as a row vector.) Let $V' \in \mathcal{V}$. For any $j \geq 0$, we have
	\begin{align*}
		\rho_{j+1}(V') &= \E[\unit\{V_{i+j+1} = V'\}\,|\,\mathcal{F}_{i}]\\
		&=\E\left[ \unit\{V_{i+j+1} = V'\}\sum_{V \in \mathcal{V}}\unit\{V_{i+j} = V\}\,\Big\vert\,\mathcal{F}_{i}\right]\\
		&=\E\left[ \E\left[  \unit\{V_{i+j+1} = V'\}\sum_{V \in \mathcal{V}}\unit\{V_{i+j} = V\}\,\Big\vert\,\mathcal{F}_{i+j}\right] \,\Big\vert\, \mathcal{F}_{i}\right]\\
		&=\E\left[\sum_{V \in \mathcal{V}}\unit\{V_{i+j} = V\} \E\left[  \unit\{V_{i+j+1} = V'\}\,\vert\,\mathcal{F}_{i+j}\right] \,\Big\vert\, \mathcal{F}_{i}\right]\\
		&=\E\left[\sum_{V \in \mathcal{V}}\unit\{V_{i+j} = V\} \left(Q(V_{i+j}, V') + O\left(\frac{\omega}{n-i} \right)\right) \,\Big\vert\, \mathcal{F}_{i}\right]\\
		&=\E\left[ \sum_{V \in \mathcal{V}}\unit\{V_{i+j} = V\}Q(V,V')\,\Big\vert\,\mathcal{F}_{i}\right] + O\left(\frac{\omega}{n-i} \right)\\
		&=(\rho_{j}Q)(V')+ O\left(\frac{\omega}{n-i} \right).
	\end{align*}
	From the above, we may write
	\begin{equation}
		\rho_{j+1} = \rho_{j}Q + v_{j},
	\end{equation}
	where $v_{j}$ is an error vector with $\norm{v_{j}}_{1} = O\left(\frac{\omega}{n-i} \right)$, uniformly in $j$. We claim that 
	\begin{equation}\label{eq:claim}
		\norm{\rho_{j} - \rho_{0}Q^{j}}_{1} \leq \sum_{j'=0}^{j}\norm{v_{j'}}_{1}.
	\end{equation}
	Since $\sum_{j'=0}^{j}\norm{v_{j'}}_{1} \leq\sum_{j'=0}^{\omega}\norm{v_{j'}}_{1} = O\left( \frac{\omega^{2}}{n-i}\right)$ for all $j$, the main result follows from (\ref{eq:claim}). Obviously the claim is true for $j=0$. Suppose the statement holds for some $j$. Then,
	\begin{align*}
		\norm{\rho_{j+1} - \rho_{0}Q^{j+1}}_{1} &= \norm{\rho_{j}Q + v_{j+1} - \rho_{0}Q^{j+1}}_{1}\\
		&\leq \norm{(\rho_{j}-\rho_{0}Q^{j})Q}_{1} + \norm{v_{j+1}}_{1}.
	\end{align*}
	We have
	\begin{align*}
		\norm{(\rho_{j}-\rho_{0}Q^{j})Q}_{1} = \norm{Q^{\top}(\rho_{j}-\rho_{0}Q^{j})^{\top}}_{1} \leq \norm{Q^{\top}}_{1} \norm{(\rho_{j}-\rho_{0}Q^{j})}_{1}.
	\end{align*}
	Since $Q$ is a stochastic matrix, the columns of $Q^{\top}$ sum to $1$, and hence $\norm{Q^{\top}}_{1} = 1$ by \eqref{eq:norm_identities}. Thus,
	$$
		\norm{(\rho_{j}-\rho_{0}Q^{j})Q}_{1} \leq \norm{\rho_{j}-\rho_{0}Q^{j}}_{1} \leq \sum_{j'=0}^{j}\norm{v_{j'}}_{1}
	$$
	by the inductive hypothesis. Thus, we conclude (\ref{eq:claim}) holds for all $j =0,\dots,\omega$.
\end{proof}

\section{The differential equations method}\label{sec:DEs}

\subsection{Notation and parameters}

Recall from Section~\ref{sec:burn_in} that after the burn-in phase, vertices $n, 1, 2,\dots, i_{0}$ are properly $3$-coloured and have had their edges partially revealed by the matching process. In this section we will analyze \textbf{Sudoku} on vertices $i_{0}+1, \dots, i_1$. (Recall that $i_{0} = \frac{n}{\log\log n}$ and $i_1 = n - \frac {n}{\log \log n}$.) We would like to understand the average trajectories of the variables $X_{k}(i)$ over short intervals of length $\omega$. To this end, given some $\omega$, we recall the definitions of the translated and rescaled variables from Section~\ref{sec:random_matching_process}:
$$
	\wt{X}_{k}(i) = \frac{1}{\omega}X_{k}(i_{0} + i\omega)
$$
and 
$$
	\wt{X}(i) = \wt{X}_{1}(i) + \wt{X}_{2}(i) + \wt{X}_{3}(i) = \frac{1}{\omega}X(i_{0} + i\omega)
$$
for $k \in \{1,2,3\}$ and integers $i$ with $0 \leq i \leq \frac{i_{1} - i_{0}}{\omega}$. 

Define $N = \frac{n}{\omega}$ and $t_{i} = \frac{i_{0} + i  \omega}{n} = \frac{i_{0}}{n} + \frac{i}{N}$. Recall the function $x(t) = t(1-t)$. Our goal is to show that the random variables $\wt{X}_{k}(i)$ all remain quite close to $\frac{N}{3}x(t_{i})$ as $i$ ranges from $0$ to $i_{\max} = \frac{i_{1} - i_{0}}{\omega} = N\left(1 - \frac{2}{\log \log n} \right)$. To this end, we will define a suitable error function $\vep(t)$ on $(0,1)$ for which it holds that $|\wt{X}_{k}(i) - \frac{N}{3}x(t_{i})| < \frac{N}{3}\vep(t_{i})$ for all $0 \leq i \leq i_{\max}$ a.a.s. We generally follow the roadmap in the excellent tutorial \cite{bennett2022gentle}.

The method requires a number of parameters, most of which must be chosen with some care. Below is a catalogue: 
	\begin{align*}
		\omega &= n^{1/3}\\
		N &= \frac{n}{\omega} = n^{2/3} \\
            i_{\max} &=  \frac{i_{1} - i_{0}}{\omega} = N\left(1 - \frac{2}{\log \log n} \right) \\
		t_{i}&= \frac{i_{0} + i\omega}{n}  = \frac{i_{0}}{n} + \frac{i}{N} \quad \text{ for } i = 0,1,\dots,i_{\max} \\
		C &= 54C_{\text{cond}} + 2 \\
		\vep(t) &= \frac{\log^{3}n}{n^{1/3}(1-t)^{C}} \quad \text{ for }t \in [0,1)\\
		x(t) &= t(1-t)
	\end{align*}
Note that we have
$$
    i_{0} + i_{\max}\omega = i_{1} = n-\frac{n}{\log \log n }
$$
so $i_{0} + i \omega < n$ for all $i = 0,1,\dots,i_{\max}.$

For $i = 0,1,\dots,i_{\max}$, let $\wt{\mathcal{E}}_{i}$ be the event that
$$
	\left\vert \wt{X}_{k}(i') - \frac{N}{3}x(t_{i'}) \right\vert \leq \frac{N}{3}\vep(t_{i'}) \quad \text{ for all }0 \leq i' \leq i \text{ and } k \in \{1,2,3\}.
$$
We remark that $\wt{\mathcal{E}}_{i}$ implies $|\wt{X}(i') - Nx(t_{i'})| \leq N\vep(t_{i}')$ for all $0 \leq i' \leq i$. Additionally, recall  
$$
	\wt{\mathcal{F}}_{i} = \mathcal{F}_{i_{0} + i\omega} \quad \text{ for } i = 0,1,\dots, i_{\max}.
$$
To make computations less cumbersome, we will use the following convention throughout this section: for quantities $A, B, C$ we write $A = B \pm C$ to indicate that $A\in[B-C, B+C]$. Thus the event $\wt{\mathcal{E}_{i}}$ says that $\wt{X}_{k}(i') = \frac{N}{3}(x(t_{i'}) \pm \vep(t_{i'}))$ for all $0 \leq i' \leq i$. 

The upper limit $i_{\max}$ ensures that $\vep(t_{i}) = o(x(t_{i}))$ holds for all $0 \leq i \leq i_{\max}$, and indeed it is easy to show the following: 
\begin{equation}\label{eq:vep_bound2}
	\frac{\vep(t_{i})}{x(t_{i})} = o\left( \frac{\log^{4}n}{n^{1/3}} \right) \quad \text{ for }i = 0,1,\dots,i_{\text{max}}.
\end{equation}

\subsection{The main result: $\wt{\mathcal{E}}_{i_{\max}}$ holds a.a.s}

To apply the differential equations method, we will need to compute the expected one-step changes 
$$
	\E[\Delta \wt{X}_{k}(i)\,|\,\wt{\mathcal{F}}_{i}]=\E[\wt{X}_{k}(i+1) - \wt{X}_{k}(i)\,|\,\wt{\mathcal{F}}_{i}]
$$
for $i = 0,1,\dots, i_{\max}-1$. In order to accurately compute the one-step change at time $i$, we will need to assume that $\wt{\mathcal{E}}_{i}$ holds; this ensures that the $\wt{X}_{k}(i)$'s are all close to $\frac{\wt{X}(i)}{3}$, enabling us to use the Markov chain of Section~\ref{sec:markov_chain} for approximations. Thus our proof is ultimately inductive. 

We will eventually show the following.

\begin{lemma}\label{lem:one_step_change}
	There is a constant $c \in (0,1)$ such that for each $k \in \{1,2,3\}$ and $i = 0,1,\dots,i_{\max}-1$
	\begin{align*}
		\E\left[\unit_{\wt{\mathcal{E}}_{i}}\Delta \wt{X}_{k}(i)\,|\,\wt{\mathcal{F}}_{i}\right] &= \frac{\unit_{\wt{\mathcal{E}}_{i}}}{3}\left(1 -  \frac{2x(t_{i})}{1-t_{i}} \pm c\vep'(t_{i})\right)\\
		&=\frac{\unit_{\wt{\mathcal{E}}_{i}}}{3}\left(1 - 2t_{i} \pm c\vep'(t_{i})\right).
	\end{align*}
\end{lemma}

Proving Lemma~\ref{lem:one_step_change} is the main technical challenge of this section. Before beginning the proof, we show how this allows us to conclude our target result.

\begin{theorem}\label{thm:diff_eq}
	The event $\wt{\mathcal{E}}_{i_{\max}}$ holds a.a.s. That is, a.a.s., 
	$$
		\left\vert \wt{X}_{k}(i) - \frac{N}{3}x(t_{i}) \right\vert \leq \frac{N}{3}\vep(t_{i}) \quad \text{ for all }0 \leq i \leq i_{\max} \text{ and } k \in \{1,2,3\}.
	$$
\end{theorem}

\begin{proof}
	For $k \in \{1,2,3\}$ and $i = 1,2,\dots,i_{\max}$, define
	$$
		\wt{X}_{k}^{+}(i) = \begin{cases}
			\wt{X}_{k}(i) - \frac{N}{3}\left(x(t_{i}) + \vep(t_{i}) \right) \quad&\text{if }\wt{\mathcal{E}}_{i-1}\text{ holds}\\
			\wt{X}_{k}^{+}(i-1) \quad&\text{otherwise}
		\end{cases}
	$$
	and $\wt{X}_{k}^{+}(0) = \wt{X}_{k}(0) - \frac{N}{3}\left(x(t_{0}) + \vep(t_{0}) \right)$. Similarly, let

	$$
		\wt{X}_{k}^{-}(i) = \begin{cases}
			\wt{X}_{k}(i) - \frac{N}{3}\left(x(t_{i}) - \vep(t_{i}) \right) \quad&\text{if }\wt{\mathcal{E}}_{i-1}\text{ holds}\\
			\wt{X}_{k}^{-}(i-1) \quad&\text{otherwise}
		\end{cases}
	$$
    for $i = 1,2,\dots,i_{\max}$ and $\wt{X}_{k}^{-}(0) = \wt{X}_{k}(0) - \frac{N}{3}(x(t_{0}) - \vep(t_{0})).$ Note that the we have the following equivalence:
    \begin{equation}\label{eq:event_equivalence}
        \wt{\mathcal{E}_{i}} \quad\text{ if and only if }\quad \wt{X}_{k}^{+}(i) \leq 0 \text{ and }\wt{X}_{k}^{-}(i) \geq 0
    \end{equation}
    We will show that $\wt{X}_{k}^{+}$ is a submartingale and $\wt{X}_{k}^{-}$ is a supermartingale. We remark that it suffices to show $\E\left[ \unit_{\wt{\mathcal{E}}_{i}}\Delta \wt{X}_{k}^{+}(i)\,|\,\wt{\mathcal{F}}_{i} \right] \leq 0$ and $\E\left[ \unit_{\wt{\mathcal{E}}_{i}}\Delta \wt{X}_{k}^{-}(i)\,|\,\wt{\mathcal{F}}_{i} \right] \geq 0$, as the increments $\Delta \wt{X}_{k}^{+}(i)$ and $\Delta \wt{X}_{k}^{-}(i)$ are $0$ by definition on the complement of $\wt{\mathcal{E}}_{i}$.
    
    Let us make two preliminary computations: for any $i$, we have	
	$$
		\frac{N}{3}(x(t_{i+1}) - x(t_{i})) = \frac{N}{3}\left(\frac{1}{N} - \frac{2t_{i}}{N} - \frac{1}{N^{2}} \right) = \frac{1}{3}\left(1 - 2t_{i} - \frac{1}{N} \right).
	$$
	Further, for $t \in (0,1)$ we have $\vep'(t) = \frac{C\log^{3}n}{n^{1/3}(1-t)^{C+1}}$ and $\vep''(t) = \frac{(C+1)C\log^{3}n}{n^{1/3}(1-t)^{C+2}}.$ By Taylor's theorem, for $t < 1 - \frac{1}{N}$,  
	\begin{equation}\label{eq:taylor1}
		\frac{N}{3}\left( \vep\left(t + \frac{1}{N} \right) - \vep(t) \right) = \frac{\vep'(t)}{3} + O\left( \frac{\vep''(\xi)}{N} \right)
	\end{equation}
	where $\xi$ is a point in the interval $\left[t, t + \frac{1}{N}\right]$. (Since $\vep''$ is increasing in $t$, the expression (\ref{eq:taylor1}) remains valid when $\xi$ is replaced with $t + \frac{1}{N}$.) If we additionally assume $t \leq t_{i_{\max}-1}$, then we may compute
	\begin{align*}
		\frac{\vep''\left(t + \frac{1}{N}\right)}{N} &= O\left(\frac{\vep'\left(t + \frac{1}{N}\right)}{n^{2/3}\left(1 - (t + \frac{1}{N})\right)} \right)\\
		&= O\left(\frac{\vep'\left(t + \frac{1}{N} \right)\log \log n }{n^{2/3}} \right)\\
        &=O\left( \frac{\vep'(t) \log\log n}{n^{2/3}} \right)\\
        &=o(\vep'(t)).
	\end{align*}
	(For the penultimate line, we use that $\vep'\left(t + \frac{1}{N} \right) = O(\vep'(t))$ uniformly for $t \geq t_{0}$, which follows from a simple computation.) Thus, from the above bound and (\ref{eq:taylor1}), for $0 \leq i \leq i_{\max}-1$ we have
	
	$$
		\frac{N}{3}\left(\vep(t_{i+1}) - \vep(t_{i})\right) = \frac{\vep'(t_{i})}{3} + o\left(\vep'(t_{i})\right).
	$$
	
	Now, using our computations and Lemma \ref{lem:one_step_change}, for $0 \leq i \leq i_{\max}-1$, 
	\begin{align*}
		\E\left[\unit_{\wt{\mathcal{E}}_{i}}\Delta \wt{X}_{k}^{+}(i)\,|\,\wt{\mathcal{F}}_{i} \right] &= \E\left[\unit_{\wt{\mathcal{E}}_{i}}\Delta \wt{X}_{k}(i)\,|\,\wt{\mathcal{F}}_{i}\right] - \unit_{\wt{\mathcal{E}}_{i}}\bigg(\frac{N}{3}\left(x(t_{i+1}) - x(t_{i}) + \vep(t_{i+1}) - \vep(t_{i}) \right)\bigg)\\
		&\leq \unit_{\wt{\mathcal{E}}_{i}}\bigg( \frac{1-2t_{i}}{3} + c\frac{\vep'(t_{i})}{3} - \frac{N}{3}\left(x(t_{i+1}) - x(t_{i}) + \vep(t_{i+1}) - \vep(t_{i}) \right)\bigg)\\
		&= \unit_{\wt{\mathcal{E}}_{i}}\bigg(\left(\frac{c -1}{3} + o(1)\right)\vep'(t_{i}) + O\left(\frac{1}{N}\right)\bigg)\\
		&\leq 0.
	\end{align*}
    
	We use that $\frac{1}{N} = o(\vep'(t_{i}))$ for all $i$ in the last line, and also assume that $n$ is sufficiently large. Thus $\wt{X}_{k}^{+}$ is a submartingale. A symmetric computation gives that $\wt{X}_{k}^{-}$ is a supermartingale.
	
	 In order to apply large deviation bounds, we also remark that the increments of $\wt{X}_{k}^{+}$ and $\wt{X}_{k}^{-}$ are bounded in absolute value: for $i = 0,1,\dots,i_{\max}-1$,
	$$
		| \Delta \wt{X}_{k}^{+}(i) |,\,| \Delta \wt{X}_{k}^{-}(i) | \leq |\Delta \wt{X}_{k}(i)| + \frac{N}{3}|x(t_{i+1}) - x(t_{i})|+ \frac{N}{3}|\vep(t_{i+1}) - \vep(t_{i})| = O(1).
	$$
	By Lemmas \ref{lem:burn_in2} and \ref{lem:burn_in3}, we have that a.a.s. 
        $$
            X_{k}(i_{0}) = \frac{n x(i_{0}/n)}{3} + O\left(\sqrt{ i_{0} \log\log n} \right) = \frac{n x(t_{0})}{3} + O(\sqrt{n}).
        $$
    It follows that, a.a.s., 
	\begin{align*}
	    \wt{X}_{k}^{+}(0) &= \frac{1}{\omega}X_{k}(i_{0}) - \frac{N}{3}\left(x(t_{0}) + \vep(t_{0}) \right)\\
        &= \frac{n x(t_{0})}{3\omega} + O\left(\frac{\sqrt{n}}{\omega}\right) - \frac{n}{3\omega}\left(x(t_{0}) + \vep(t_{0}) \right)\\
        &= O\left(\frac{\sqrt{n}}{\omega}\right) - \frac{n}{3\omega}\vep(t_{0})\\
        &= O(n^{1/6}) - \frac{n^{1/3}\log^{3}n}{3} \\
        &= -(1+o(1))\frac{n^{1/3}\log^{3}n}{3}.
	\end{align*}
	We similarly get $\wt{X}_{k}^{-}(0) = (1+o(1))\frac{n^{1/3}\log n}{3}$ a.a.s. In particular $\wt{\mathcal{E}}_{0}$ holds (with room to spare) a.a.s.\ by the equivalence (\ref{eq:event_equivalence}). 
	
	Now, if the event $\wt{\mathcal{E}}_{i}$ fails at some $i = 1,\dots,i_{\max}$, then we have either $\wt{X}_{k}^{+}(i') > 0$ for all $i' \geq i$, or $\wt{X}_{k}^{-}(i') < 0$ for all $i' \geq i$. Therefore,
	\begin{equation}\label{eq:martingale_bound1}
		\prob(\wt{\mathcal{E}_{i}} \text{ fails for some }i = 1,2,\dots,i_{\max}) \leq \sum_{k=1}^{3}\prob(\wt{X}_{k}^{+}(i_{\max}) > 0) + \prob(\wt{X}_{k}^{-}(i_{\max}) < 0).
	\end{equation}
	Since $\wt{X}_{k}^{+}$ is a submartingale with increments bounded by a constant, Proposition~\ref{prop:azuma} (Azuma's inequality) gives 
	\begin{align*}
		\prob(\wt{X}_{k}^{+}(i_{\max}) > 0) &= \prob(\wt{X}_{k}^{+}(i_{\max}) - \wt{X}_{k}^{+}(0) > -\wt{X}_{k}^{+}(0))\\
		&= \prob\left(\wt{X}_{k}^{+}(i_{\max}) - \wt{X}_{k}^{+}(0) > (1+o(1))\frac{n^{1/3}\log^{3}n}{3}\right)\\
		&\leq \exp\left\{ -\Omega\left(\frac{n^{2/3}\log^{6}n}{N}\right) \right\}\\
		&\leq \exp\left\{ -\Omega(\log^{6}n)\right\}\\
		&=o(1)
	\end{align*}
	Again using Proposition \ref{prop:azuma}, we can similarly bound $\prob(\wt{X}_{k}^{-}(i_{\max}) < 0) \leq \exp\left\{ -\Omega(\log^{6}n)\right\}$. From these bounds and (\ref{eq:martingale_bound1}), the desired result follows. 
\end{proof}

\subsection{Expected one-step changes}

We now turn to proving Lemma~\ref{lem:one_step_change}. We define two auxiliary random variables which make our computations simpler. For $i \in [n]$, $h \in [n-i]$, and $k \in \{1,2,3\}$, define
	$$
		H_{k}(i,h) = \sum_{j=1}^{i}\unit\{c(j) = k \text{ and }p(j) \in \{i+1,\dots,i+h\}\}.
	$$
We think of $H_{k}(i,h)$ as the number of unsaturated vertices in $[i]$ of colour $k$ which are ``hit" by the edges revealed from time $i+1$ to $i+h$; the hit vertices stop contributing to $X_{k}$ after time $i+h$. Let
	$$
		B_{k}(i,h) = \sum_{j=1}^{h}\unit\{c(i+j) = k \text{ and }p(i+j) > i+h\}.
	$$
Then, $B_{k}(i,h)$ counts the number of unsaturated vertices of colour $k$ which are ``born" (and ``survive'') from time $i+1$ to $i+h$. We may write 
	\begin{eqnarray}\label{eq:Delta_X_decomposition}
		\E\left[\Delta \wt{X}_{k}(i)\,|\,\wt{\mathcal{F}}_{i}\right] &=& \frac{1}{\omega}\E\left[X_{k}(i_{0}+(i+1)\omega) - X_{k}(i_{0}+i\omega)\,|\,\wt{\mathcal{F}}_{i}\right] \nonumber \\
		&=&\frac{1}{\omega}\E\left[-H_{k}(i_{0}+i\omega, \omega) + B_{k}(i_{0}+i\omega, \omega)\,|\,\wt{\mathcal{F}}_{i}\right] \nonumber \\
		&=&-\frac{1}{\omega}\E\left[H_{k}(i_{0}+i\omega, \omega) \,|\,\wt{\mathcal{F}}_{i}\right] + \frac{1}{\omega}\E\left[B_{k}(i_{0}+i\omega, \omega) \,|\,\wt{\mathcal{F}}_{i}\right].
	\end{eqnarray}
We perform the computations of the two terms on the right-hand side of (\ref{eq:Delta_X_decomposition}) separately, beginning with $\frac{1}{\omega}\E[H_{k}(i_{0} + i\omega, \omega)\,|\,\wt{\mathcal{F}}_{i}]$, which is easier. 
	
\begin{lemma}\label{lem:K_lemma} 
For $i =0,1,\dots,i_{\max}-1$, 
	$$
		\frac{1}{\omega}\E[\unit_{\wt{\mathcal{E}}_{i}}H_{k}(i_{0} + i\omega, \omega)\,|\,\wt{\mathcal{F}}_{i}] = \unit_{\wt{\mathcal{E}}_{i}}\bigg(\frac{t_{i}}{3} \pm \frac{\vep'(t_{i})}{3C}\bigg).	
	$$
\end{lemma}

\begin{proof}
	Let $I_{i}$ be the interval $\{i_{0} + i\omega + 1, \dots, i_{0} + (i+1)\omega\}$. For any vertex $j \leq i_{0} + i\omega$, if $j$ is unsaturated at time $i_{0}+i\omega$, then the probability that $j$ is still unsaturated at time $i_{0}+(i+1)\omega$ is exactly
	$$
		\prod_{h=0}^{\omega-1}\left(1 - \frac{1}{n-(i_{0}+i\omega+h)} \right) = \frac{(n-i_{0}-i\omega-1)_{\omega}}{(n-i_{0}-i\omega)_{\omega}} = \frac{n-i_{0}-(i+1)\omega}{n-i_{0}-i\omega}
	$$
	and hence 
	$$
		\prob(p(j) \in I_{i}\,|\,\wt{\mathcal{F}}_{i}) = \unit\{p(j) > i_{0}+i\omega\}\frac{\omega}{n-i_{0}-i\omega}.
	$$
	We get,
	\begin{align*}
		\frac{1}{\omega}\E\left[H_{k}(i_{0}+i\omega, \omega)\,|\,\mathcal{F}_{i\omega}\right] &= \frac{1}{\omega}\sum_{j=1}^{i_{0}+i\omega}\E\left[\unit\{c(j) = k \text{ and }p(j) \in I_{i}\}\,|\,\wt{\mathcal{F}}_{i}\right] \\
		&= \frac{1}{\omega}\sum_{j=1}^{i_{0}+i\omega}\unit\{c(j) = k\}\E\left[ \unit\{p(j) \in I_{i}\}\,|\,\wt{\mathcal{F}}_{i}\right] \\
		&= \frac{1}{\omega}\sum_{j=1}^{i_{0}+i\omega}\unit\{c(j) = k\}\unit\{p(j) > i_{0}+i\omega\}\frac{\omega}{n-i_{0}-i\omega} \\
		&= \frac{X_{k}(i_{0}+i\omega)}{n-i_{0}-i\omega} = \frac{\wt{X}_{k}(i)}{N(1-t_{i})}.
	\end{align*}
	On the event $\wt{\mathcal{E}}_{i}$, we have
	$$
		\frac{\wt{X}_{k}(i)}{N(1-t_{i})} = \frac{x(t_{i}) \pm \vep(t_{i})}{3(1-t_{i})} = \frac{t_{i}}{3} \pm \frac{\vep(t_{i})}{3(1-t_{i})} = \frac{t_{i}}{3} \pm \frac{\vep'(t_{i})}{3C}.	
	$$
	This finishes the proof of the lemma.
\end{proof}

Computing $\E[B_{k}(i_{0}+i\omega, \omega)\,|\,\wt{\mathcal{F}}_{i}]$ more challenging.

\begin{lemma}\label{lem:B_lemma} For $i = 0 ,1,\dots, i_{\max}$,
	$$
		\E[\unit_{\wt{\mathcal{E}}_{i}}B_{k}(i_{0} + i \omega,\omega)\,|\,\wt{\mathcal{F}}_{i}] = \unit_{\wt{\mathcal{E}}_{i}}\bigg(\frac{1-t_{i}}{3} \pm (1+o(1))\frac{54C_{\text{cond}} + 1}{3C}\vep'(t_{i}) \bigg).
	$$
\end{lemma}

\begin{proof}
	Define 
	$$
		B_{k}'(i,h) = \sum_{j=1}^{h}\unit\{c(i+j) = k \text{ and }p(i+j) > i+j\}
	$$
	and note that for any $i,h$ we have 
	\begin{equation}\label{eq:bk_bound1}
		B_{k}'(i,h) -e(i,h)  \leq B_{k}(i,h) \leq B_{k}'(i,h),
	\end{equation}
	where $e(i,h)$ is the number of matching edges which join pairs of vertices in $\{i+1,\dots, i+h\}$. To ease notation, going forward we will write $B_{k} = B_{k}(i_{0} + i\omega, \omega)$ and $B_{k}' = B_{k}'(i_{0} + i\omega, \omega)$.
	
	Conditioned on $\wt{\mathcal{F}}_{i}$, we expect 
		$$
			\binom{\omega}{2}\frac{n-i_{0}-i\omega - X(i_{0}+i\omega)}{(n-i_{0}-i\omega)(n-i_{0}-i\omega -1)} \leq \frac{\omega^{2}}{2(n-i_{0}-i\omega -1)} = O\left( \frac{1}{n^{1/3}(1-t_{i})} \right)
		$$
	edges in the interval $\{i_{0}+i\omega + 1, \dots, i_{0}+(i+1)\omega\}$. Thus, by (\ref{eq:bk_bound1}), 
	\begin{equation}\label{bk_bound2}
		 \frac{1}{\omega}\E[B_{k}\,|\,\wt{\mathcal{F}}_{i}] =  \frac{1}{\omega}\E[B_{k}'\,|\,\wt{\mathcal{F}}_{i}] + O\left( \frac{1}{n^{2/3}(1-t_{i})} \right).
	\end{equation}
	To bound $\E[B_{k}'\,|\,\mathcal{F}_{i\omega}]$, we use a comparison with the Markov chain introduced in Section~\ref{sec:markov_chain}. For $k \in \{1,2,3\}$, let
	$$
		q_{k} = \frac{X_{k}(i_{0}+i\omega)}{n-i_{0}-i\omega} = \frac{\wt{X}_{k}(i)}{N(1-t_{i})}
	$$ 
	and define $q = q_{1} + q_{2} + q_{3}$. Let $Q$ be the transition matrix with parameters $q_{1},q_{2},q_{3}$ and denote the stationary distribution of $Q$ by $\pi$. We also let $Q_{\text{bal}}$ be the transition matrix with all parameters equal to $\frac{q}{3}$, and denote its stationary distribution by $\pi_{\text{bal}}.$ For any $j'$, let $V_{j'} \in \mathcal{V}$ be the type assigned to vertex $j'$ by \textbf{Sudoku} at time $j'$, and for $j \in 0,\dots,\omega$ let $\rho_{j}$ be the distribution on $\mathcal{V}$ defined by 
	$$
		\rho_{j}(V) = \prob(V_{i_{0}+i\omega+j} = V\,|\,\wt{\mathcal{F}}_{i}).
	$$
	Our plan is to approximate $\rho_{j}$ by $\rho_{0}Q^{j}$. (Note that conditioned on $\wt{\mathcal{F}}_{i}$, $\rho_{0}$ assigns probability $1$ to $V_{i_{0}+i\omega}$.) There are multiple sources of error which arise in the approximation, and we introduce some additional parameters to account for them:
		\begin{align*}
			\delta_{\text{comp}} &:= \max_{0 \leq j \leq \omega}\norm{\rho_{j} - \rho_{0}Q^{j}}_{\infty} \quad \text{ the error incurred by approximating }\rho_{j}\text{ by }\rho_{0}Q^{j}\\
			\delta_{\text{mix}} &:= \frac{t_{\text{mix}}\left(\frac{1}{N}\right)}{\omega} \quad \text{ the error incurred by waiting for }Q\text{ to mix to within }\frac{1}{N}\text{ of stationary}\\
			\delta_{\text{cond}} &:= \norm{\pi - \pi_{\text{bal}}}_{\infty} \quad \text{ the error incurred by approximating }\pi\text{ by }\pi_{\text{bal}}.
		\end{align*}
	We will deal with each of these in turn, and then perform the actual calculations. The parameters $i_{0}$ and $\omega$ and the error function $\vep(t)$ were chosen somewhat carefully so that the following situation obtains: $\delta_{\text{comp}}$ and $\delta_{\text{mix}}$ will be ``small" errors, in the sense that it will suffice to absorb them into $O(\cdot)$. The term $\delta_{\text{cond}}$, however, will be ``large", and we will need to pay attention to constants when dealing with it.
	
	The term $\delta_{\text{comp}}$ is handled by Lemma~\ref{lem:comparison}. Observe that for any $0 \leq i \leq i_{\max}-1$
    $$
        \frac{\omega}{n - i_{0} - i\omega} \leq \frac{\omega}{n-i_{0} - i_{\max}\omega}  = O\left(\frac{\log \log n}{n^{2/3}}\right) = o(1).
    $$
    Thus we may apply Lemma~\ref{lem:comparison} to get
	\begin{equation}\label{eq:delta_comp}
		\delta_{\text{comp}} = O\left(\frac{\omega^{2}}{n-i_{0} - i\omega} \right) = O\left(\frac{\omega}{N(1-t_{i})} \right) = O\left( \frac{1}{n^{1/3}(1-t_{i})} \right).
	\end{equation}
	To handle the other two error terms, we will apply Lemmas~\ref{lem:mixing_time} and~\ref{lem:condition_number}, respectively. These both require the parameters $q_{k}$ to satisfy $(1 - \gamma)\frac{q}{3} < q_{k} < (1 + \gamma)\frac{q}{3}$ for some value of $\gamma \in (0,1)$ (and Lemma~\ref{lem:mixing_time} requires $\gamma$ to be sufficiently small). By conditioning on the event $\wt{\mathcal{E}}_{i}$, we can compute a value of $\gamma$ for which these bounds hold. Indeed, given $\wt{\mathcal{E}}_{i}$, we have that 
	$$
		q_{k} = \frac{\wt{X}_{k}(i)}{N(1-t_{i})} = \frac{x(t_{i}) \pm \vep(t_{i})}{3(1-t_{i})}
	$$
	for all $k$ and 
	$$
		q = \frac{\wt{X}(i)}{N(1-t_{i})} = \frac{x(t_{i}) \pm \vep(t_{i})}{1-t_{i}}.
	$$
	We may write
	$$
		\frac{q_{k}}{q/3} = \frac{x(t_{i})\pm \vep(t_{i})}{x(t_{i}) \mp \vep(t_{i})} = 1 \pm \frac{2\vep(t_{i})}{x(t_{i})} + O\left(\left(\frac{\vep(t_{i})}{x(t_{i})}\right)^{2}\right).
	$$
(Above, we use that $\vep(t_{i}) = o(x(t_{i}))$ by (\ref{eq:vep_bound2}).) Thus, when the event $\wt{\mathcal{E}}_{i}$ holds, we may assign $\gamma = \frac{2\vep(t_{i})}{x(t_{i})} + O\left(\left(\frac{\vep(t_{i})}{x(t_{i})} \right)^{2}\right)$. Note that $\gamma = O\left(\frac{\vep(t_{i})}{x(t_{i})} \right) = o\left( \frac{\log^{4}}{n^{1/3}}\right) = o(1)$ uniformly in $0 \leq i \leq i_{\max}-1$ by (\ref{eq:vep_bound2}), so in particular we are free to apply Lemma~\ref{lem:mixing_time}, provided $n$ is large enough.
	
	Lemma~\ref{lem:mixing_time} gives $t_{\text{mix}}\left(\frac{1}{N} \right) = O\left(\frac{\log N}{q}\right)$. On $\wt{\mathcal{E}}_{i}$, we have
	\begin{align*}
		q &= \frac{x(t_{i}) \pm \vep(t_{i})}{1-t_{i}}\\ 
            &= \left(1 + O\left(\frac{\vep(t_{i})}{x(t_{i})} \right) \right)\frac{x(t_{i})}{1-t_{i}}\\
            &= (1+o(1))t_{i}\\
            &\geq (1+o(1))t_{0}\\
            &= (1+o(1))\frac{1}{\log\log n}.
	\end{align*}
	Therefore, $t_{\text{mix}}\left(\frac{1}{N}\right) = O(\log N \log \log n) = O(\log^{2}n)$. So, when $\wt{\mathcal{E}}_{i}$ holds, 
	\begin{equation}\label{eq:delta_mix}
		\delta_{\text{mix}} = O\left(\frac{\log^{2}n}{\omega} \right) = O\left(\frac{\log^{2}n}{n^{1/3}} \right).
	\end{equation}
	Lemma~\ref{lem:condition_number} gives that $\delta_{\text{cond}} \leq 3 C_{\text{cond}} \gamma q$. On $\wt{\mathcal{E}}_{i}$, we compute
	$$
		\gamma q = \frac{2\vep(t_{i})}{1-t_{i}} + O\left( \frac{\vep(t_{i})^{2}}{x(t_{i})(1-t_{i})}\right) = \frac{2\vep'(t_{i})}{C} + o(\vep'(t_{i}))
	$$
	and so 
	\begin{equation}\label{eq:delta_cond}
		\delta_{\text{cond}} = (1+o(1))\frac{6C_{\text{cond}}\vep'(t_{i})}{C}.
	\end{equation}
	With the error terms (\ref{eq:delta_comp}), (\ref{eq:delta_mix}), and (\ref{eq:delta_cond}) in hand, we are ready to compute $\E[\unit_{\wt{\mathcal{E}}_{i}}B_{k}'\,|\,\wt{\mathcal{F}}_{i}]$.  Let $F^{(k)} =\{A_{f}^{(k)}, B_{f}^{(lk)}, B_{f}^{(mk)}\} \subset \mathcal{V}$. For $j =0,1,\dots,\omega$ we can express
	$$
		\unit\{c(i_{0} + i\omega+j) = k\text{ and }p(i_{0} + i\omega + j) > i_{0} + i\omega + j\} = \unit\{V_{i_{0} + i\omega+j} \in F^{(k)}\}.
	$$
	Write $t_{\text{mix}} = t_{\text{mix}}\left(\frac{1}{N} \right)$. To bound $B_{k}'$, we use
	\begin{equation}\label{eq:Bk_bound1}
		\sum_{j=t_{\text{mix}}+ 1}^{\omega}\unit\{V_{i_{0} +i\omega+j} \in F^{(k)}\} \leq B_{k}'\leq t_{\text{mix}}+ \sum_{j=t_{\text{mix}}+1}^{\omega}\unit\{V_{i_{0} + i\omega+j} \in F^{(k)}\}.
	\end{equation}
	The sums above are empty if $t_{\text{mix}} \geq \omega$, in which case we evaluate them to $0$. Assuming that $t_{\text{mix}} < \omega$, for any $j$ with $t_{\text{mix}}+1 \leq j \leq \omega$,	
	\begin{align*}
		\E[\unit\{A_{i_{0} + i\omega+j} \in F^{(k)}\}\,|\,\wt{\mathcal{F}}_{i}] &= \rho_{j}(F^{(k)})\\
		&= (\rho_{0}Q^{j})(F^{(k)}) + O\left(\delta_{\text{comp}} \right)\\
		&= \pi(F^{(k)}) + O\left(\delta_{\text{comp}} + \frac{1}{N}\right)\\
		&= \pi_{\text{bal}}(F^{(k)}) \pm 3\delta_{\text{cond}} + O\left(\delta_{\text{comp}} + \frac{1}{N} \right)\\
		&= \frac{1-q}{3} \pm 3\delta_{\text{cond}} + O\left(\delta_{\text{comp}} +\frac{1}{N} \right)\\
		&= \frac{1}{3}\left(1 - \frac{\wt{X}(i)}{N(1-t_{i})} \right) \pm 3\delta_{\text{cond}} + O\left(\delta_{\text{comp}} + \frac{1}{N} \right).
	\end{align*}
	In the penultimate line we use $\pi_{\text{bal}}(F^{(k)}) = \frac{1-q}{3}$, which follows from an easy computation using the stationary distribution for $\pi_{\text{bal}}$ (\ref{eq:stationary_dist}). Combining the above with (\ref{eq:Bk_bound1}) yields the following:	
	\begin{equation}\label{eq:Bk_bound2}
		 \frac{1}{\omega}\E[B_{k}'\,|\,\wt{\mathcal{F}}_{i}] = \frac{1}{3}\left(1 - \frac{\wt{X}(i)}{N(1-t_{i})} \right) \pm 3\delta_{\text{cond}} + O\left(\delta_{\text{comp}} + \delta_{\text{mix}}+\frac{1}{N} \right).
	\end{equation}
	On the event $\wt{\mathcal{E}}_{i}$, using (\ref{eq:delta_cond}) the main term in (\ref{eq:Bk_bound2}) satisfies 	
	\begin{align*}
		 \frac{1}{3}\left(1 - \frac{\wt{X}(i)}{N(1-t_{i})} \right) + 3\delta_{\text{cond}} &\leq \frac{1}{3}\left(1 - \frac{x(t_{i}) - \vep(t_{i})}{1-t_{i}} \right) + (1+o(1))\frac{18C_{\text{cond}}\vep'(t_{i})}{C}\\
		 &=\frac{1}{3}\left(1 - \frac{x(t_{i})}{1-t_{i}} \right) + \frac{\vep'(t_{i})}{3C} + (1+o(1))\frac{18C_{\text{cond}}\vep'(t_{i})}{C}\\
		 &=\frac{1-t_{i}}{3} + (1+o(1))\frac{54C_{\text{cond}} + 1}{3C}\vep'(t_{i}).
	\end{align*}
	In the same way, we get the lower bound
	$$
		 \frac{1}{3}\left(1 - \frac{\wt{X}(i)}{N(1-t_{i})} \right) - 3\delta_{\text{cond}} \geq  \frac{1-t_{i}}{3} - (1+o(1))\frac{54C_{\text{cond}} + 1}{3C}\vep'(t_{i}).
	$$
	Using (\ref{eq:delta_comp}) and (\ref{eq:delta_mix}), the $O(\cdot)$ error term in (\ref{eq:Bk_bound2}) is, on $\wt{\mathcal{E}}_{i}$, 
	$$
		\delta_{\text{comp}} + \delta_{\text{mix}}+\frac{1}{N} = O\left( \frac{1}{n^{1/3}(1-t_{i})} + \frac{\log^{2}n}{n^{1/3}} + \frac{1}{n^{2/3}} \right) = o\left( \vep'(t_{i})\right).
	$$
	In total,
	$$
		\frac{1}{\omega}\E\left[\unit_{\wt{\mathcal{E}}_{i}}B_{k}' \right] =  \unit_{\wt{\mathcal{E}}_{i}}\bigg(\frac{1-t_{i}}{3} \pm (1+o(1))\frac{54C_{\text{cond}} + 1}{3C}\vep'(t_{i})\bigg)
	$$
	To finish, recall from (\ref{bk_bound2}) that $\frac{1}{\omega}\E[B_{k}\,|\,\wt{\mathcal{F}}_{i}] = \frac{1}{\omega}\E[B_{k}'\,|\,\wt{\mathcal{F}}_{i}] + O\left(\frac{1}{n^{2/3}(1-t_{i})} \right)$. The $O(\cdot)$ term here is clearly also $o\left(\vep'(t_{i}) \right)$, and thus we have
	$$
		\frac{1}{\omega}\E\left[\unit_{\wt{\mathcal{E}}_{i}}B_{k}\right] = \unit_{\wt{\mathcal{E}}_{i}}\bigg(\frac{1-t_{i}}{3} \pm (1+o(1))\frac{54C_{\text{cond}} + 1}{3C}\vep'(t_{i}) \bigg).
	$$
	This finishes the proof of the lemma. 
\end{proof}

We are now able to prove Lemma \ref{lem:one_step_change}.

\begin{proof}[Proof of Lemma \ref{lem:one_step_change}]
    From (\ref{eq:Delta_X_decomposition}) and Lemmas~\ref{lem:K_lemma} and~\ref{lem:B_lemma}, we may write
	$$
		\E\left[\unit_{\wt{\mathcal{E}}_{i}}\Delta \wt{X}_{k}(i)\,|\,\wt{\mathcal{F}}_{i}\right] = \unit_{\wt{\mathcal{E}}_{i}}\bigg(\frac{1-2t_{i}}{3} \pm (1+o(1))\frac{54C_{\text{cond}} +1 }{3C}\vep'(t_{i})\bigg)
	$$
	for any $0 \leq i \leq i_{\max} -1$. Since $C = 54C_{\text{cond}} + 2$, there is some constant $c \in (0,1)$ so that $(1+o(1))\frac{54C_{\text{cond}} + 1}{C} < c$ for $n$ sufficiently large. The result follows.
\end{proof}

\section{Proof of Theorem~\ref{thm:main_theorem}}

\begin{proof}[Proof of Theorem~\ref{thm:main_theorem}]
    Recall that we add all $\frac{2n}{\log \log n}$ vertices from the burn-in and completion phases to $S$. Using Lemma~\ref{lem:size_bound}, we then get
    $$
        |S| \leq \frac{1}{2}|\mathcal{B}_{C}| + |\mathcal{B}_{U}^{(c)}| + 2|\mathcal{B}_{U}^{(d)}| + \frac{2n}{\log \log n},
    $$
    where $\mathcal{B}_{C}$, $\mathcal{B}_{U}^{(c)}$, and $\mathcal{B}_{U}^{(d)}$ are as in Subsection~\ref{sec:correctness}. We write $\mathcal{B}_{U} = \mathcal{B}_{U}^{(c)} \cup \mathcal{B}_{U}^{(d)}$ and note $|\mathcal{B}_{U}^{(c)}| + 2|\mathcal{B}_{U}^{(d)}| \leq 2|\mathcal{B}_{U}|$. It is easy to show that $|\mathcal{B}_{U}|=o(n)$ a.a.s. Indeed, for any $j \in \{i_{0},\dots, i_{1}-1\}$ we have
    $$
        \prob\Big(j+1 \in \mathcal{B}_{U}\,|\,\mathcal{F}_{j}\Big) \leq \prob\Big(p(j+1) = \textbf{ptr}(i)+1\,|\,\mathcal{F}_{j}\Big) \leq \frac{1}{n-j}.
    $$
    Thus $\E[|\mathcal{B}_{U}|]  \leq \sum_{j=i_{0}}^{i_{1} -1}\frac{1}{n-j} \le i_{1} \cdot \frac{1}{n-i_{1}} = O\left(\log\log n \right)$. That $|\mathcal{B}_{U}| = o(n)$ a.a.s.\ then follows from Markov's inequality, and we conclude that $|\mathcal{B}_{U}^{(c)}| + 2|\mathcal{B}_{U}^{(d)}| = o(n)$.

    Now we focus on bounding $|\mathcal{B}_{C}|$. Let $j \in \{i_{0}, \dots, i_{1} - 1\}.$ Conditioned on $\mathcal{F}_{j}$, there is a single colour $k(j)$ which is ``forbidden" for $p(j+1)$ if the current run is to continue. Vertex $j+1$ is in $\mathcal{B}_{C}$ if and only if it has a forward edge, or it has a backward edge with $p(j+1) \leq \textbf{ptr}(j)$ and $c(p(j+1)) = k(j).$ Thus
    \begin{equation}\label{eq:bc_prob}
        \prob\Big(j+1 \in \mathcal{B}_{C}\,|\,\mathcal{F}_{j}\Big) \leq 1 - \frac{X(j)}{n-j} + \frac{X_{k(j)}(j)}{n-j}.
    \end{equation}
    Now, consider the largest $i$ so that $i_{0} + i \omega \leq j.$ We write 
    \begin{align*}
        \frac{X_{k(j)}(j)}{n-j} &\leq \frac{X_{k(j)}(i_{0} + i \omega) + \omega}{n-j}\\
        &= \frac{X_{k(j)}(i_{0} + i \omega)}{n-(i_{0} + i\omega)}\left(\frac{1}{1-\frac{j-(i_{0} + i \omega)}{n-(i_{0} + i\omega)}} \right) + O\left(\frac{\log \log n}{n^{2/3}} \right)\\
        &=\frac{X_{k(j)}(i_{0} + i \omega)}{n-(i_{0} + i\omega)}\left(1 + O\left(\frac{\omega}{n-(i_{0} + i\omega)} \right) \right) + O\left(\frac{\log \log n}{n^{2/3}} \right)\\
        &=\frac{\wt{X}_{k(j)}(i)}{N(1-t_{i})} + O\left( \frac{\log \log n}{n^{2/3}} \right).
    \end{align*}
    Similarly, 
    $$
        \frac{X(j)}{n-j} \geq \frac{X(i_{0} + i\omega) - \omega}{n-(i_{0} + i\omega)} = \frac{\wt{X}(i)}{N(1-t_{i})} - O\left( \frac{\log \log n}{n^{2/3}}\right).
    $$
    On the event $\wt{\mathcal{E}}_{i}$, we thus have
    \begin{align*}
        1- \frac{X(j)}{n-j} + \frac{X_{k(j)}(j)}{n-j} &\leq 1 - \frac{\wt{X}(i)}{N(1-t_{i})} + \frac{\wt{X}_{k(j)}(i)}{N(1-t_{i})} + O\left( \frac{\log \log n}{n^{2/3}} \right)\\
        &\leq 1 - \frac{x(t_{i}) - \vep(t_{i})}{1-t_{i}} + \frac{x(t_{i}) + \vep(t_{i})}{3(1-t_{i})} + O\left( \frac{\log \log n}{n^{2/3}} \right)\\
        &=1 - \frac{2t_{i}}{3} + O\left(\frac{\vep(t_{i})}{1-t_{i}} + \frac{\log \log n}{n^{2/3}} \right)\\
        &=1 - \frac{2t_{i}}{3} + O\left(\frac{\log^{4}n}{n^{1/3}} \right)\\
        &=(1+o(1))\left(1 - \frac{2t_{i}}{3} \right).
    \end{align*}
    For the penultimate line, we use that $\frac{\vep(t_{i})}{1-t_{i}} = \frac{\log^{3}n}{n^{1/3}(1-t_{i})^{C+1}} \leq \frac{\log^{3}n (\log \log n)^{C+1}}{n^{1/3}}$ for $t_{i} \leq 1 - \frac{1}{\log \log n}.$ For the last, note that $\frac{1}{3} \leq 1 - \frac{2t_{i}}{3} \leq 1$ for $t_{i} \in [0,1]$.

    Since $\wt{\mathcal{E}}_{i}$ holds for all $i = 0, 1, \dots, i_{\max} = \frac{n}{\omega}\left(1 - \frac{1}{\log \log n} \right)$ a.a.s.\ by Theorem~\ref{thm:diff_eq}, we may conclude from the above and (\ref{eq:bc_prob}) that for all $j \in \{i_{0} + 1, \dots, i_{1}\}$, we have 
    $$
    \prob\Big(j+1 \in \mathcal{B}_{C}\,|\,\mathcal{F}_{j}\Big) \leq (1+o(1))\left(1 - \frac{2t_{i}}{3}\right),
    $$
    where $i$ is the largest index so that $i_{0} + i\omega \leq j$.

    The sum $\sum_{j=i_{0} + 1}^{i_{1}}\unit\{j \in \mathcal{B}_{C}\}$ is thus stochastically dominated by a sum of independent $\{0,1\}$-random variables
    $$
        Z = \sum_{i=0}^{i_{\max}-1}\sum_{h=1}^{\omega}Z_{h}^{(i)},
    $$
    where $\prob(Z_{h}^{(i)} = 1) = (1+o(1))\left(1 - \frac{2t_{i}}{3} \right)$ for all $h = 1, 2, \dots, \omega$, uniformly in $i$. We compute
    \begin{align*}
        \E[Z] &= (1+o(1))\omega\sum_{i=0}^{i_{\max}-1}\left(1 - \frac{2t_{i}}{3} \right)\\
        &=(1+o(1))\omega \sum_{i=0}^{i_{\max}-1}\left(1 - \frac{2}{3}\cdot\frac{i_{0} + i\omega}{n} \right)\\
        &=(1+o(1))\omega i_{\max}\left(1 - \frac{2i_{0}}{3n} \right) - \frac{2\omega}{3n}\sum_{i=0}^{i_{\max}-1}i\\
        &=(1+o(1))\left(n - \frac{2\omega}{3n}\frac{i_{\max}^{2}}{2}\right)\\
        &=(1+o(1))\frac{2n}{3}.
    \end{align*}
    It follows from Proposition~\ref{prop:chernoff} that $Z \leq (1+o(1))\E[Z] = (1+o(1))\frac{2n}{3}$ a.a.s. Since $Z$ stochastically dominates $|\mathcal{B}_{C}|$, we may conclude $\frac{1}{2}|\mathcal{B}_{C}| \leq (1+o(1))\frac{n}{3}$ a.a.s., completing the proof.
\end{proof}

\section{Conclusion}

In this paper, we have shown that $s(\mathcal{G}_{n,3})\leq (1+o(1))\frac{n}{3}$ a.a.s., and have additionally made the conjecture that $s(\mathcal{G}_{n,3}) = (1+o(1))\frac{n}{4}$ a.a.s., which is (asymptotically) smallest possible for cubic graphs on $n$ vertices. Resolving this conjecture is an interesting direction for future research. Another one is to study $s(\mathcal{G}_{n,d})$ for $d \geq  4$---a problem which at present appears challenging. 

We considered the following crude technique for upper bounding the Sudoku number of a graph $G$ with chromatic number $k$. Start with any $k$-coloring of $G$ using colour classes $\{1,2,\dots,k\}$. Iteratively move vertices from class $k$ to another available class until every remaining vertex has a neighbour in each class $i$ for $i = 1,\dots,k-1$. Then the union of classes $1$ through $k-1$ is a Sudoku set for $G$, and we have the bound 
\begin{equation}\label{eq:independence_bound}
    s(G) \leq (\chi(G)-1)\alpha(G),
\end{equation}
where $\alpha(G)$ is the size of the largest independent set in $G$.

Using the best available bound $\alpha(\mathcal{G}_{n,4}) \leq 0.41635$ a.a.s.\ (see \cite{mckay1987lnl}) and the fact that $\chi(\mathcal{G}_{n,4})=3$ a.a.s.\ (\cite{shi2007colouring4}), \eqref{eq:independence_bound} gives $s(\mathcal{G}_{n,4}) \leq 0.8327 n$ a.a.s. The next value of $d$ for which $\chi(\mathcal{G}_{n,d})$ is known explicitly is $d = 6$, where we have $\chi(\mathcal{G}_{n,d}) = 4$ a.a.s.\ \cite{shi2007colouring}. However, it has been shown that $\mathcal{G}_{n,6}$ has an independent set of size at least $0.33296n$ a.a.s.\ \cite{duckworth2009large}, meaning that any bound using \eqref{eq:independence_bound} for $d = 6$ will be very close to $n$. In fact, currently available upper bounds, such as $0.35799n$, are well above $n/3$ (see for instance \cite{duckworth2009large}) and thus, are far from providing any useful bound on $s(\mathcal{G}_{n,6})$. For large $d$, \eqref{eq:independence_bound} is too weak to say anything nontrivial about $s(\mathcal{G}_{n,d})$. Indeed, for $d \to \infty$ sufficiently slowly with respect to $n$, it is known that $\chi(\mathcal{G}_{n,d}) = (1+o(1))\frac{n}{\alpha(\mathcal{G}_{n,d})} = (1+o(1))\frac{d}{2\log d}$ a.a.s.\ \cite{frieze1992independence}. 

For small $d$, one could attempt to construct a Sudoku set algorithmically as we did for $d = 3$. In \cite{shi2007colouring4} and \cite{shi2007colouring}, Shi and Wormald analyze an algorithm which efficiently colours $\mathcal{G}_{n,d}$ a.a.s.\ for $d \in \{4,5,\dots,10\}$, using exactly $\chi(\mathcal{G}_{n,d})$ colours for $d = 4$ and $d=6$ and at most $\chi(\mathcal{G}_{n,d}) + 1$ colours for the other values of $d$. It is conceivable that their algorithm could be adapted to produce a Sudoku set along with the colouring, although this task appears to be difficult.

\bibliographystyle{plain}
\bibliography{refs}

\end{document}